\documentclass[12pt]{amsart}
\usepackage[T2A]{fontenc}
\usepackage[utf8]{inputenc}   
\usepackage[russian, english]{babel}
\usepackage{enumerate}

\usepackage{float}

\usepackage{hyperref}
\usepackage{amsmath}
\usepackage{amsthm}
\usepackage{amsfonts}
\usepackage{amssymb}
\usepackage{cite}

\usepackage{mathptmx} 
\usepackage[scaled=0.90]{helvet} 
\usepackage{courier} 
\usepackage[T1]{fontenc}
\usepackage{color}

\usepackage[dvips]{graphicx}
\DeclareGraphicsExtensions{.eps}

\newtheorem{theorem}{Theorem}[section]
\newtheorem{utv*}{Proposition}
\newtheorem{hyp*}{Conjecture}
\newtheorem*{example*}{Example}
\newtheorem{lemma}[theorem]{Lemma}
\newtheorem{corollary}[theorem]{Corollary}

\newtheorem*{th*}{Theorem}
\newtheorem{prop}[theorem]{Proposition}
\theoremstyle{definition}

\newtheorem*{rem*}{Remark}
\newtheorem{defin}[theorem]{Definition}

\def\sli{\sum\limits}
\def\ili{\int\limits}

\def\R{\mathbb{R}}

\def\ep{\varepsilon}

\def\s{\delta}

\newcommand{\diam}{\operatorname{diam}}

\newcommand{\PP}{\mathcal{P}}

\def\cyr{\fontencoding{OT2}\fontfamily{wncyr}\selectfont}
\DeclareTextFontCommand{\textcyr}{\cyr}

{\end{list}}

\newcounter{vremennyj}

\def\H{\mathcal{H}}

\renewcommand\S{\mathbb{S}}

\title[Covering and separation properties of Chebyshev points]{Covering and separation of Chebyshev points for non-integrable Riesz potentials}
\date{\today}
\begin{document}
\author{A. Reznikov}
\address{Center for Constructive Approximation, Department of Mathematics, Vanderbilt University}
\email{aleksandr.b.reznikov@vanderbilt.edu}
\author{E. B. Saff}
\thanks{The research of A. Reznikov and E. B. Saff was supported, in part, by the National Science Foundation grants DMS-1516400}
\email{edward.b.saff@vanderbilt.edu}
\author{A. Volberg}
\address{Department of Mathematics, Michigan State University}
\email{volberg@math.msu.edu}
\thanks{The research of A. Volberg was supported, in part, by the National Science Foundation grant DMS-1600065}
\begin{abstract}
For Riesz $s$-potentials $K(x,y)=|x-y|^{-s}$, $s>0$, we investigate separation and covering properties of $N$-point configurations $\omega^*_N=\{x_1, \ldots, x_N\}$ on a $d$-dimensional compact set $A\subset \R^\ell$ for which the minimum of $\sum_{j=1}^N K(x, x_j)$ is maximal. Such configurations are called $N$-point optimal Riesz $s$-polarization (or Chebyshev) configurations. For a large class of $d$-dimensional sets $A$ we show that for $s>d$ the configurations $\omega^*_N$ have the optimal order of covering. Furthermore, for these sets we investigate the asymptotics as $N\to \infty$ of the best covering constant. For these purposes we compare best-covering configurations with optimal Riesz $s$-polarization configurations and determine the $s$-th root asymptotic behavior (as $s\to \infty$) of the maximal $s$-polarization constants. In addition, we introduce the notion of ``weak separation'' for point configurations and prove this property for optimal Riesz $s$-polarization configurations on $A$ for $s>\textup{dim}(A)$, and for $d-1\leqslant s < d$ on the sphere $\mathbb{S}^d$.
\end{abstract}
\maketitle

\section{Introduction}
Suppose $A$ is a compact subset of a Euclidean space $\R^\ell$ and $\omega_N=\{x_1, \ldots, x_N\}\subset A$ is a {\it multiset} (or an {\it $N$-point configuration}); i.e., a set of points with possible repetitions and cardinality $\# \omega_N = N$, counting multiplicities. For a positive number $s$ we put
$$
P_s(A;\omega_N):=\inf_{y\in A} \sum_{j=1}^N \frac{1}{|y-x_j|^s}.
$$
Then the {\it $N$-th s-polarization (or Chebyshev) constant of $A$} is defined by
$$
\PP_{s}(A;N):=\sup_{\omega_N\subset A} P_s(A;\omega_N).
$$
We note that since $A$ is compact, there exists for each $N\in \mathbb{N}$ a configuration $\omega_N^*=\{x_1^*, \ldots, x_N^*\}$ and a point $y^*$ such that
\begin{equation}\label{optconfig}
\PP_s(A;N)=P_s(A;\omega_N^*) = \sli_{j=1}^N \frac{1}{|y^*-x_j^*|^s}. 
\end{equation}
We call $\omega_N^*$ an {\it optimal (or extremal) Riesz $s$-polarization configuration} or simply an {\it optimal configuration}.

From an applications prospective, the maximal polarization problem, say on a compact surface (or body), can be viewed as the problem of determining the smallest number of sources (injectors) of a substance together with their optimal locations that can provide a required saturation of the substance at every point of the surface (body). 

The general notion of polarization (or Chebyshev constants) for potentials was likely first introduced by Ohtsuka \cite{Ohtsuka1967}. Further investigations of the asymptotic behavior as $N\to \infty$ of polarization constants as well as the asymptotic behavior of optimal configurations appear, for example, in \cite{Ambrus2013}, \cite{Erdelyi2013}, \cite{Farkas2006}, \cite{Farkas2008}, \cite{borodachov2014asymptotics}, \cite{Simanek2015}, \cite{Borodachov2016}, \cite{borodachov2016optimal}, \cite{reznikov2016minimum}.

The following result is a special case of a theorem due to Borodachov, Hardin, Reznikov and Saff \cite{borodachov2016optimal} (see also \cite{Borodachov2016}). It describes the asymptotic behavior of optimal configurations for the case of non-integrable Riesz kernels on $A$. Here and throughout we denote by $\H_d$ the Hausdorff measure on $\R^\ell$, $d\leqslant \ell$, normalized by $\H_d([0,1]^d)=1$.
\begin{theorem}\label{thborodachovetal}
Suppose $A$ is a compact $C^1$-smooth $d$-dimensional manifold, embedded in $\R^\ell$ with $d\leqslant \ell$, and $\H_d(\partial A)=0$, where $\partial A$ denotes the boundary of $A$. If $s>d$, then there exists a positive finite constant $\sigma_{s,d}$ that does not depend on $A$ such that
\begin{equation}\label{definofsigma}
\lim_{N\to \infty} \frac{\PP_s(A; N)}{N^{s/d}} = \frac{\sigma_{s, d}}{\H_d(A)^{s/d}}.
\end{equation}
Moreover, if $\{\omega_N^*\}_{N=1}^\infty$ is any sequence of optimal configurations satisfying \eqref{optconfig}, then the normalized counting measures $\mu_N^*$ for the multisets $\omega_N^*$ satisfy
$$
\mu^*_N:=\frac{1}{N}\sli_{x\in \omega^*_N} \delta_{x} \stackrel{*}{\to} \mu,
$$
where $\stackrel{*}{\to}$ denotes convergence in the weak$^*$ topology, and $\mu$ is the uniform measure on $A$; i.e., for any Borel set $B\subset \R^\ell$
$$
\mu(B)=\frac{\H_d(B\cap A)}{\H_d(A)}.
$$
\end{theorem}
In other words, in the limit, optimal polarization configurations $\omega_N^*$ for non-integrable Riesz potentials are uniformly distributed in the weak$^*$ sense. In this paper we study more distributional properties of optimal configurations $\omega_N^*$. In particular, we investigate their separation, their covering (or mesh) radius, and their connection to the ``best covering problem'' for the set $A$. 
\begin{defin}
Let $A$ be a compact subset of a Euclidean space. For any $N$-point configuration $\omega_N\subset A$, the {\it separation constant} of $\omega_N$ is defined by
$$
\delta(\omega_N):=\min_{i\not=j}|x_i-x_j|
$$
and the {\it covering radius} of $\omega_N$ is defined by
\begin{equation}
\rho_A(\omega_N):=\max_{y\in A}\min_{x\in \omega_N} |y-x|.
\end{equation}
The {\it best $N$-point covering radius for $A$} $\rho_A(N)$ is given by
\begin{equation}\label{defbestcovering}
\rho_A(N):=\min_{\omega_N\subset A} \rho_A(\omega_N),
\end{equation}
where the minimum is taken over all $N$-point configurations $\omega_N\subset A$. 
\end{defin}
In approximation theory (for example, in interpolation by splines), the separation constant $\delta(\omega_N)$ often measures ``stability'' of approximation, while the covering radius $\rho_A(\omega_N)$ is involved in bounds for the error of the approximation (see, e.g., \cite{breger2016points}). Quasi-uniform sequences; i.e., sequences $\{\omega_N\}_{N=2}^\infty$ for which the ratios $\rho_A(\omega_N)/\delta(\omega_N)$ are bounded from above, appear, for example, in a number of applications involving approximation by radial basis functions, see, e.g., \cite{MR2740542}. Thus they play an important role in the complexity analysis for such applications. 

Regarding the asymptotic behavior of polarization constants as $s$ grows large, it is known, see \cite{Borodachov2016}, that for a fixed $N$ we have
$$
\lim_{s\to \infty}\left(\frac{\PP_s(A; N)}{N^{s/d}}\right)^{1/s} = \frac1{N^{1/d}\rho_A(N)}.
$$
However, the proof in \cite{Borodachov2016} does not guarantee that this limit is uniform in $N$; thus it does not imply any asymptotic behavior of the constants $\sigma_{s,d}$ in \eqref{definofsigma} as $s\to \infty$. One of our main results, Theorem \ref{bestcovering}, shows that for a large class of $d$-dimensional sets $A$,
\begin{equation}\label{sigmalimittt}
\lim_{s\to \infty}\left(\frac{\sigma_{s,d}}{\H_d(A)^{s/d}}\right)^{1/s}=\lim_{s\to \infty}\lim_{N\to \infty} \left(\frac{\PP_s(A; N)}{N^{s/d}}\right)^{1/s} = \frac{1}{\lim_{N\to \infty} N^{1/d} \rho_A(N)}.
\end{equation}
In the case when $A\subset \R^2$ is a compact set with $\H_2(A)>0$, it is known \cite{Kershner1939} that
$$
\lim_{N\to \infty} N^{1/2} \rho_A(N) = \frac{\sqrt{2}}{\sqrt[4]{27}} \H_2(A)^{1/2};
$$
thus from \eqref{sigmalimittt},
$$
\lim_{s\to \infty}\sigma_{s,2}^{1/s} = \frac{\sqrt[4]{27}}{\sqrt{2}}.
$$
For higher dimensions we prove that all limits in \eqref{sigmalimittt} exist.

\bigskip

We shall work primarily with the class of $d$-regular sets.

\begin{defin}\label{defregular}
A compact set $A\subset \R^\ell$ is called {\it $d$-regular} if there exist a measure $\mu$ supported on $A$ and two positive constants $c_1$ and $c_2$ such that for any $x\in A$ and any positive $r<\diam(A)$, we have 
\begin{equation}
c_1r^d\leqslant \mu(A\cap B(x, r))\leqslant c_2r^d,
\end{equation}
where $B(x,r)$ is the open ball in $\R^\ell$ with center $x$ and radius $r$.
\end{defin}
The following estimate from above for $\PP_s(A; N)$, which follows from \cite[Theorem 2.4]{Erdelyi2013} and its proof, will be useful for our investigation.
\begin{theorem}\label{botest}
If $A\subset \R^\ell$, $\ell\geqslant d$, $\H_d(A)>0$ and $s>d$, then there exists a constant $C_s>0$, that depends on $d$, $A$ and $s$ such that, for any positive integer $N$,
\begin{equation}\label{aboveest}
\PP_s(A;N)\leqslant C_s N^{s/d}.
\end{equation}
Moreover, $C_s$ can be chosen so that there exists a constant $C_0$ with the property that for large values of $s$ we have $1\leqslant (C_s)^{1/s}\leqslant C_0$.
\end{theorem}
The following immediate consequence of this theorem will be proved in Section \ref{secappendix}.
\begin{prop}\label{babyhole}
With the hypotheses of Theorem \textup{\ref{botest}}, let $\omega_N=\{x_j\}_{j=1}^N$ be a fixed $N$-point configuration on $A$. There exists a positive constant $c_s$, independent of $N$ and $\omega_N$, with the following property: if $y^*=y^*_s\in A$ is a point such that
$$
\sli_{j=1}^N \frac{1}{|y^*-x_j|^s}=\inf_{y\in A}\sli_{j=1}^N \frac{1}{|y-x_j|^s},
$$
then $|y^*-x_j|\geqslant c_s N^{-1/d}$ for each $j=1,\ldots, N$. Moreover, $c_s$ can be chosen so that $\lim_{s\to \infty}c_s^{1/s}=1$.

Furthermore, the same is true for $s\in [d-1, d)$ when $A=\mathbb{S}^d$, the $d$-dimensional unit sphere in $\mathbb{R}^{d+1}$.
\end{prop}


We next introduce the main class of sets $A$ that we will consider. 

\begin{defin}\label{strconv}
A compact set $A\subset \R^d$ is called a {\it body} if $A\not=\emptyset$ and $A=\textup{Clos}(\textup{Int}(A))$. We say that a body $A\subset \R^d$ is {\it strongly convex} if it is convex and its boundary $\partial A$ is a $(d-1)$-dimensional $C^2$-smooth manifold with non-degenerate Gaussian curvature \footnote{Such conditions appear in many problems in harmonic analysis, see, e.g., \cite{Iosevich20010}.}.
\end{defin}
This class includes the closed unit ball 
$$
\mathbb{B}^d:=\{x\in \R^d\colon |x|\leqslant 1\}$$ 
and ellipsoids 
$$
\{(x_1, \ldots, x_d)\colon x_1^2/a_1^2 + \cdots + x_d^2/a_d^2 \leqslant 1\};$$ 
however, it does not include cubes and polyhedra.

The paper is organized as follows. In Section \ref{sectionmain} we state and discuss our main results. In Section \ref{sectionconvex} we prove a `weak separation' result for strongly convex bodies. In Section \ref{sectioncube} we prove the `weak separation' for the unit cube $[0,1]^d$, and in Section \ref{sectionsphere} we prove it for the unit sphere $\S^d$ and spherical caps in $\S^d$. Further, in Section \ref{sectionmesh}, we derive a criterion for a sequence of configurations to have an optimal order of covering radius $\rho_A(\omega_N)$. We also show that configurations $\omega^*_N$ that are optimal for $\PP_s(A; N)$ satisfy this criterion if $A$ is strongly convex, a cube, a sphere, or a spherical cap. And, in Section \ref{sectionlimit}, we connect the asymptotic behavior of the constant $\sigma_{s,d}$ as $s\to \infty$ with the asymptotic behavior of the best covering radius $\rho_N(A)$, where $A$ is any of the sets just mentioned.
We prove Proposition \ref{babyhole} in Section \ref{secappendix} and in the Appendix (Section \ref{appendix}) we present equivalent definitions of best covering for the space $\R^d$.
\section{Main results}\label{sectionmain}
For strongly convex bodies $A\subset \R^d$ the separation and covering properties of extremal configurations $\omega_N^*$ for $\PP_s(A; N)$, in general, depend on the parameter $s$. Here we shall prove `weak separation' and covering properties for $s>d$.
In contrast, it is known \cite{Erdelyi2013} that for the closed $d$-dimensional unit ball $\mathbb{B}^d\subset \R^d$ and for $0<s\leqslant d-2$, the unique optimal $N$-point $s$-polarization configuration $\omega^*_N$ is $\omega^*_N = \{0, \ldots, 0\}$; thus,
$$
\delta(\omega^*_N)=0, \;\;\;\; \rho_{A}(\omega^*_N) = 1, \; \; \; \forall N. 
$$
The main reason behind this is that the function
$$
x\mapsto |x-y|^{-s}
$$
is superharmonic when $s\leqslant d-2$. 

Our first goal is to establish for the non-integrable case $s>d$ a weak-separation property in the following sense.
 
\begin{defin}
A family $\Omega$ of multisets $\omega$ from $A$, where $A\subset \R^\ell$ has Hausdorff dimension $d$, is called {\it weakly well-separated with parameter $\eta>0$} if there exists an $M\in \mathbb{N}$ such that for every $\omega\in \Omega$ and every point $z\in \R^\ell$, we have
\begin{equation}\label{mmm}
\#\big(\omega \cap B(z, \eta \cdot (\#\omega)^{-1/d})\Big) \leqslant M.
\end{equation}
\end{defin}
It is easy to see that for a $d$-regular set $A$ there exists a positive constant $C$ such that for any configuration $\omega\subset A$ we have
\begin{equation}
\delta(\omega)\leqslant C\cdot (\#\omega)^{-1/d}.
\end{equation}
If for some $\eta>0$ inequality \eqref{mmm} holds with $M=1$ for every $\omega\in \Omega$, then
$$
\delta(\omega)\geqslant \eta \cdot (\#\omega)^{-1/d};
$$
therefore, we get the optimal order of separation with respect to the cardinality of $\omega$. 

\bigskip

\begin{defin}
A set $A$ is called {\it $d$-admissible} if $A\subset \R^d$ is strongly convex, or $A=\S^d\subset \R^{d+1}$, or $A\subset \S^d$ is a spherical cap.
\end{defin}
We prove the following theorems.
\begin{theorem}\label{thsepar}
If $d\in \mathbb{N}$, $s>d$, and the set $A$ is $d$-admissible, then there exists an $\eta>0$ such that the family $\Omega=\Omega_s:=\{\omega\colon P_s(A; \omega)=\PP_s(A; \#\omega)\}$ is weakly well-separated with parameter $\eta$ and $M=2d-1$.
Moreover, $\eta=\eta_s$ can be chosen so that $\lim_{s\to \infty} \eta_s^{1/s}=1$.

The same is true for $s\in [d-1,d)$ when $A=\S^d$.
\end{theorem}
The result for strongly convex bodies is proved in Section \ref{sectionconvex}, while the results for the sphere and spherical caps are proved in Section \ref{sectionsphere}.

\begin{rem*}
If $d=1$ and $A=[0,1]$, then for every $s>1$, the family $\Omega=\Omega_s$ is weakly well-separated with some $\eta>0$ and $M=1$.
\end{rem*}

As a consequence of the proof of Theorem \ref{thsepar}, we obtain the following.
\begin{corollary}
Assume $A\subset \R^d$ is a compact set and $s>d$. For every $r>0$, there exists an $\eta>0$ that depends on $r$ with the following property: if for some $z\in A$ we have $B(z,r)\subset A$, then $\#(\omega^*_N \cap B(z, \eta N^{-1/d}))\leqslant 2d-1$, where $\omega_N^*$ is optimal for $\PP_s(A; N)$.
\end{corollary}
\begin{rem*}
As we shall show in Lemma \ref{notnearboundary}, if $A$ is strongly convex then no points from $\omega_N^*$ can lie on the boundary $\partial A$; moreover, the distance from any point in $\omega_N^*$ to $\partial A$ is at least of the order $N^{-2/d}$. 
\end{rem*}
The next theorem deals with the unit cube. For this case, our methods impose a stronger condition on the Riesz parameter $s$.
\begin{theorem}\label{thseparcube}
If $[0,1]^d\subset \R^d$, $d\geqslant 2$, denotes the unit cube and $s>3d-4$, then there exists a $\eta>0$ such that the family $\Omega=\Omega_s=\{\omega\colon P_s(A; \omega)=\PP_s(A; \#\omega)\}$ is weakly well-separated with parameter $\eta$ and $M=2d-1$.
Moreover, $\eta=\eta_s$ can be chosen so that $\lim_{s\to \infty} \eta_s^{1/s}=1$.
\end{theorem}

Regarding the covering radius of $N$-point configurations having a weak separation property we prove the following.

\begin{theorem}\label{thcovering}
Let $\ell$, $d$ and $s$ be positive integers with $\ell\geqslant d$ and $s>d$. Suppose the compact set $A\subset \R^\ell$ with $\H_d(A)>0$ is contained in some $d$-regular compact set $\tilde{A}$. If the $N$-point configuration $\omega_N\subset A$ is such that for some $\eta>0$ and $M\in\mathbb{N}$ we have $\#(B(z, \eta N^{-1/d})\cap \omega_N)\leqslant M$ for all $z\in A$, then 
\begin{equation}
\rho_A(\omega_N)=\max_{y\in A}\min_{x\in \omega_N} |y-x|\leqslant R_s N^{-1/d},
\end{equation}
where
\begin{equation}
R_s:=\left(\frac{7^s \cdot C_d\cdot M\cdot s}{5^s \cdot p_s\cdot (s- d) \cdot\eta^d}\right)^{\frac{1}{s-d}},
\end{equation}
$C_d$ is a positive constant that depends only on $d$ and $A$, and $p_s$ is any positive constant such that 
\begin{equation}\label{pspsps}
\inf_{y\in A} \sli_{x\in \omega_N} \frac{1}{|y-x|^s} \geqslant p_s N^{s/d}.
\end{equation}
\end{theorem}
From this theorem and Theorem \ref{thsepar} we deduce the following.
\begin{corollary}\label{coveringclopen}
If the set $A$ is $d$-admissible and $s>d$, then there exists a positive constant $R_s$ such that for any $N$-point configuration $\omega^*_N$ that is optimal for $\PP_s(A; N)$, we have $\rho_A(\omega^*_N)\leqslant R_s N^{-1/d}$. Moreover, there exists a positive constant $R_0$ such that for large values of $s$ we have $R_s\leqslant R_0$.

The same is true if $A=[0,1]^d$ and $s> 3d-4$.
\end{corollary}

Corollary \ref{coveringclopen} implies that if $A$ is an $d$-admissible set or a unit cube, then $\rho_A(N)\leqslant R_s N^{-1/d}$ for some positive constant $R_s$. On the other hand, it is easy to see that in this case, for some positive constant $b$, we have $\rho_A(N)\geqslant bN^{-1/d}$. Fine estimates on the constant $R_s$ for large values of $s$ result in the following theorem dealing with the asymptotic behavior of $\PP_s(A; N)^{1/s}$ as $s\to \infty$.

\begin{theorem}\label{bestcovering}
Suppose the set $A$ is $d$-admissible or $A=[0,1]^d$.
Then with $\sigma_{s,d}$ as defined in Theorem \ref{thborodachovetal}, the following limits exist as positive real numbers and satisfy	
\begin{equation}\label{princess}
\lim_{s\to \infty} \left(\frac{\sigma_{s,d}}{\H_d(A)^{s/d}}\right)^{1/s}=\lim_{s\to \infty} \left(\lim_{N\to \infty}\frac{\PP_s(A; N)}{N^{s/d}}\right)^{1/s} = \frac{1}{\lim_{N\to \infty} N^{1/d}\rho_A(N)}.
\end{equation}
In particular, taking $A=[0,1]^d$ we obtain
\begin{equation}
\lim_{s\to \infty} \sigma_{s,d}^{1/s} = \frac{1}{\lim_{N\to\infty} N^{1/d}\rho_{[0,1]^d}(N)} = \left(\frac{V_d}{\Gamma_d}\right)^{1/d},
\end{equation}
where the constant $\Gamma_d$ is the optimal covering density \footnote{The problem of finding $\Gamma_d$ is known in \cite{ConwSlBook} as ``finding the thinnest covering of $\R^d$.''} of the space $\R^d$ (see \cite[Chapter 2]{ConwSlBook} and Section \ref{appendix}) and $V_d:=\H_d(\mathbb{B}^d)=\pi^{d/2}/\Gamma(d/2+1)$.
\end{theorem}
We remark that $\Gamma_1=1$ and $\Gamma_2=2\pi/\sqrt{27}$.

A consequence Theorem \ref{bestcovering} is that, in the limit as $s\to \infty$, the covering radius of optimal Riesz $s$-polarization configurations become asymptotically best possible.
\begin{corollary}
Suppose the set $A$ is $d$-admissible or $A=[0,1]^d$. For every $s>3d-4$, let $\omega_N^s$ be an $N$-point configuration such that $\PP_s(A; N)=P_s(A; \omega_N^s)$. Then
$$
\lim_{s\to\infty}\lim_{N\to\infty}N^{1/d}\rho_A(\omega_N^s) = \lim_{N\to \infty}N^{1/d}\rho_A(N).
$$
\end{corollary}

\bigskip

\section{Weak separation for strongly convex bodies}\label{sectionconvex}
In what follows, we always assume $s>d$ and $A\subset \R^d$ is a strongly convex body. 
By $\overline{B(x,r)}$ we denote the closure of $B(x,r)$ and $I_{d-1}$ denotes the $(d-1)\times (d-1)$ identity matrix. Furthermore, the $j$'th coordinate of a point $x\in \R^d$ will be denoted by $x(j)$; we also denote by $x'$ the $(d-1)$-dimensional vector that consists of the first $d-1$ coordinates of $x$; thus, $x=(x', x(d))$. By $e_1, \ldots, e_d$ we denote the canonical basis in $\R^d$. If we have a $d\times d$ matrix $M$, we put	
$$
(Mx, x):=(Mx)\cdot x, \;\; \; \; \; x\in \R^d.
$$

To establish Theorem \ref{thsepar} we begin with two lemmas about the behavior of extremal configurations for $\PP_s(A; N)$ near the boundary $\partial A$.
\begin{lemma}\label{notnearboundary}
There exists a constant $b_s>0$ with the following property: for all $N\geqslant 1$, if $\omega^*_N$ is an extremal configuration for $\PP_s(A; N)$ and $x\in \omega^*_N$, then $\textup{dist}(x, \partial A)>b_s N^{-2/d}$ . Moreover, $b_s$ can be chosen so that $\lim_{s\to \infty} b_s^{1/s} =1$.
\end{lemma}
\begin{rem*}
Let $x_\partial \in \partial A$ and make a rotation so that in the neighborhood $B(x_\partial, r)$ the manifold $\partial A$ is given by $\{(x', x(d))\colon x(d)=f(x')\}$ with $\nabla f(x_\partial')=0$ and the matrix $\textup{d}^2 f(x)$ is non-positive for $x\in \partial A\cap \overline{B(x_\partial, r)}$ (this can be done since $A$ is convex). Moreover, $r$ can be chosen sufficiently small so that
$$
\overline{B(x_\partial, r)} \cap A= \overline{B(x_\partial, r)}\cap \{x\colon x(d)\leqslant f(x')\}.
$$
We notice that the Gaussian curvature of $\partial A$ at $x_\partial$ is equal to the product of eigenvalues of the matrix $\textup{d}^2 f(x_\partial')$. Since in Theorem \ref{thsepar} we assume the Gaussian curvature is non-zero, the manifold $\partial A$ is compact and $C^2$-smooth and $\text{d}^2 f\leqslant 0$, we deduce that there exists a constant $C_A>0$ such that $\text{d}^2f(x')\leqslant -C_A I_{d-1}$ for every $x\in B(x_\partial, r)$, where $C_A$ does not depend on $x_\partial$.
\end{rem*}
\begin{proof}[Proof of Lemma \ref{notnearboundary}]
Take a point $x_\partial\in \partial A$ for which $|x-x_\partial|=\textup{dist}(x, \partial A)$. We can make a rotation and assume $x=x_\partial - cN^{-2/d}\cdot e_{d}$. We show that this is impossible if $c$ is sufficiently small.

Let $f$ be the function from the above remark. For a small positive number $\ep$ consider a point
$$
\tilde{x}:=x-\ep e_d \in A
$$
and a configuration $\widetilde{\omega}_N:=(\omega^*_N\setminus \{x\})\cup \{\tilde{x}\}$. Consider a point $\tilde{y}$ such that
$$
P(A; \widetilde{\omega}_N) = \sli_{\tilde{x}_j\in \widetilde{\omega}_N} \frac{1}{|\tilde{y}-\tilde{x}_j|^s}.
$$
Since $\omega^*_N$ is an extremal configuration, we have
$$
P_s(A; \omega^*_N)\geqslant P_s(A; \widetilde\omega_N),
$$
which after utilizing the definition of $\widetilde\omega_N$ implies
$$
|\tilde{y}-x|\leqslant |\tilde{y}-\tilde{x}|.
$$
Using that $\tilde{x}=x-\ep e_d$, we get
$$
\tilde{y}(d)-x(d)\geqslant -\ep/2, 
$$
or
$$
\tilde{y}(d)\geqslant x(d)-\ep/2 =x_\partial(d)-cN^{-2/d}-\ep/2.
$$
Since $\ep$ is an arbitrarily small number, we can assume $\ep/2 \leqslant cN^{-2/d}$. Then we obtain
$$
\tilde{y}(d)\geqslant x_\partial(d)-2cN^{-2/d}.
$$
On the other hand, since $A$ is a convex set, and the plane $\{z\in \R^d\colon z(d)=x_\partial(d)\}$ is tangent to $\partial A$, we have $\tilde{y}(d)\leqslant x_\partial(d)$. 

We now estimate the diameter of the set 
$$
S(N, c):=\{y\in A\colon x_\partial(d)-2cN^{-2/d}\leqslant y(d) \leqslant x_\partial(d)\}.
$$ 
Since $A$ is strongly convex, we obviously have $A\cap \{z\in \R^d\colon z(d)=x_\partial(d)\}=\{x_\partial\}$. Thus, $\textup{diam}(S(N,c))\to 0$ as $c\to 0$. If $c$ is chosen small enough, then $S(N, c) \subset B(x_\partial, \eta)\cap A$ for some $\eta>0$. 
Therefore, if $y$ belongs to $S(N, c)$, then for some $\xi\in B(x_\partial, \eta)$ we have
\begin{multline}
x_\partial(d)-2cN^{-2/d}\leqslant y(d)\leqslant f(y') = f(x_\partial') + \frac12 (\textup{d}^2 f(\xi')(y'-x_\partial'), (y'-x_\partial'))\\  \leqslant x_\partial(d)-\frac{C_A}{2}\cdot |y'-x_\partial'|^2,
\end{multline}
which implies
\begin{equation}
|y'-x_\partial '|^2\leqslant \frac{4c}{C_A}\cdot N^{-2/d};
\end{equation}
thus, for a suitable constant $C_B$,
$$
|y-x_\partial|^2\leqslant \frac{4c}{C_A}\cdot N^{-2/d}+4c^2 N^{-4/d}\leqslant C_B\cdot c\cdot N^{-2/d}.
$$
Therefore, since $\ep\leqslant 2cN^{-2/d}$,
$$
|\tilde{y}-\tilde{x}|\leqslant |\tilde{y}-x_\partial|+2cN^{-2/d} \leqslant \tau\cdot \sqrt{c}\cdot N^{-1/d} 
$$
for some constant $\tau$ that does not depend on $s$. For $c$ sufficiently small, this inequality contradicts Proposition \ref{babyhole} and so the lemma follows.
\end{proof}
In the next lemma we show that if $x\in A$ is close to $\partial A$ in one direction, then its distance in orthogonal directions can be estimated from below.
\begin{lemma}\label{vovsestoroni}
Let $\omega^*_N$ be an extremal configuration for $\PP_s(A; N)$ and $x\in \omega^*_N$. Assume $\tau$ is a sufficiently small positive number that does not depend on $N$. If $\textup{dist}(x, \partial A)=|x-x_\partial|$ with $x-x_\partial$ parallel to $e_d$, then the estimate $|x-x_\partial|< \tau N^{-1/d}$ implies $x\pm \tau N^{-1/d}e_j \in A$ for every $j=1, \ldots, d-1$.
\end{lemma}
\begin{proof}
Again let $f$ be as in the above remark. Arguing as in the preceding lemma, we see that we need to show that $|x-x_\partial|<\tau N^{-1/d}$ implies $x(d)\leqslant f(x'\pm\tau N^{-1/d}e_j')$. Notice that since $x\in \omega^*_N$, we know that $|x-x_\partial|>cN^{-2/d}$ for some constant $c$. We apply the Taylor formula again:
\begin{equation}
f(x'\pm\tau N^{-1/d}e_j')=x_\partial(d)+\frac{\tau^2 N^{-2/d}}2 (\textup{d}^2f(\xi') e'_j, e'_j).
\end{equation}
Since the boundary $\partial A$ is compact and smooth, we can always assume $\textup{d}^2f(\xi')>-CI_{d-1}$ for some positive constant $C$. Thus,
$$
f(x'\pm\tau N^{-1/d}e_j') \geqslant x_\partial(d)-C\tau^2 N^{-2/d} \geqslant x(d) + (c-C\tau^2)N^{-2/d}\geqslant x(d)
$$
if $\tau$ is sufficiently small. 
\end{proof}

We are ready to prove Theorem \ref{thsepar}.
\begin{proof}[Proof of Theorem \ref{thsepar} for a strongly convex set $A$]
We argue by contradiction. Suppose there exists small number $\eta>0$ and an extremal configuration $\omega^*_N=\{x_1, \ldots, x_N\}$ such that $\{x_1, \ldots, x_{2d}\}\subset B(z, \eta N^{-1/d})$. Consider
$$
\hat{x}:=\frac{x_1+\cdots+x_{2d}}{2d}\in A.
$$
Since $\hat{x}\in B(z, \eta N^{-1/d})$, we have $|x_j-\hat{x}|\leqslant 2\eta N^{-1/d}$ for every $j=1,\ldots, 2d$.

Fix a small number $\tau>\eta$. We will choose it later to be a multiple of $\eta$. Set $\ep:=\tau N^{-1/d}$. We consider two cases.
\paragraph{\bf{Case 1:}} $\textup{dist}(\hat{x}, \partial A)\geqslant \ep$.
Define $2d$ points as follows:
\begin{alignat*}{2}
\tilde{x}_1&:=\hat{x}-\ep e_1,& \;\;\;\; \tilde{x}_2&:=\hat{x}+\ep e_1,\\
\tilde{x}_3&:=\hat{x}-\ep e_2,& \;\;\;\; \tilde{x}_4&:=\hat{x}+\ep e_2,\\
 &\qquad \qquad\qquad \qquad\ldots\\
\tilde{x}_{2d-1}&:=\hat{x}-\ep e_d,& \;\;\; \tilde{x}_{2d}&:=\hat{x}+\ep e_d.
\end{alignat*}
Since $\textup{dist}(\hat{x}, \partial A)\geqslant\ep$, these points belong to $A$. Define $\widetilde{\omega}_N:=\{\tilde{x}_1, \ldots, \tilde{x}_{2d}, \tilde{x}_{2d+1}, \ldots, \tilde{x}_{N}\}$, where $\tilde{x}_j:=x_j$ for $j\geqslant 2d+1$.
Let $\tilde{y}$ be such that
\begin{equation}\label{ytildedefin}
P_s(A; \widetilde\omega_N)=\sli_{j=1}^N\frac{1}{|\tilde{y}-\tilde{x}_j|^s}.
\end{equation}
We have
$$
\sli_{j=1}^N\frac{1}{|\tilde{y}-\tilde{x}_j|^s} \leqslant \PP_s(A; N)=P_s(A; \omega^*_N) \leqslant \sli_{j=1}^N\frac{1}{|\tilde{y}-x_j|^s},
$$
and thus
\begin{equation}\label{mainineqconvex}
\sli_{j=1}^{2d}\frac{1}{|\tilde{y}-\tilde{x}_j|^s} \leqslant \sli_{j=1}^{2d}\frac{1}{|\tilde{y}-x_j|^s},
\end{equation}
Set $f(x):=|\tilde{y}-x|^{-s}$.  Then, from the Taylor formula about $\hat{x}$, we have for $x\in\{x_1,\ldots, x_{2d}\}$
$$
f(x)=f(\hat{x}) + s\frac{(\tilde y-\hat{x})\cdot (x-\hat{x})}{|\tilde y-\hat{x}|^{s+2}} + \frac{1}2 \cdot\left(-s\cdot \frac{|x-\hat{x}|^2}{|\tilde y-\xi|^{s+2}}+s(s+2)\cdot \frac{((\tilde y-\xi)\cdot (x-\hat{x}))^2}{|\tilde y-\xi|^{s+4}}\right),
$$
for some $\xi=\xi(x)\in B(\hat{x}, |x-\hat{x}|)$. From Proposition \ref{babyhole} we know that $|\tilde{y}-\tilde{x}_1|\geqslant c_s N^{-1/d}$. Without loss of generality we assume $\tau<c_s/2$, and so
\begin{equation}\label{trumpumpum}
|\tilde{y}-\hat{x}|=|\tilde{y}-\tilde{x}_1+\ep e_1|\geqslant (c_s-\tau)N^{-1/d}\geqslant (c_s/2)\cdot N^{-1/d},
\end{equation}
and
$$
|\tilde y-\xi|\geqslant |\tilde y-\hat{x}|-|\hat{x}-\xi|\geqslant |\tilde y - \hat{x}|-|x-\hat{x}|\geqslant |\tilde y - \hat{x}|-2\eta N^{-1/d} \geqslant (1-4\eta/c_s)|\tilde y-\hat{x}|.
$$
Therefore, for every $j=1,\ldots, 2d$ we have
$$
f(x_j)\leqslant f(\hat{x})+s\frac{(\tilde y-\hat{x})\cdot(x_j-\hat{x})}{|\tilde y-\hat{x}|^{s+2}} + \frac{2s(s+3)\eta^2 N^{-2/d}(1-4\eta/c_s)^{-s-2}}{|\tilde y-\hat{x}|^{s+2}}.
$$
Summing these inequalities over $j$ and recalling that $x_1+\cdots+x_{2d}=2d\hat{x}$ yields
\begin{equation}
\sli_{j=1}^{2d}\frac{1}{|\tilde y-x_j|^s} \leqslant 2d\cdot f(\hat{x}) +  \frac{4sd (s+3)\cdot \eta^2 N^{-2/d}\cdot(1-4\eta/c_s)^{-s-2}}{|\tilde y-\hat{x}|^{s+2}}.
\end{equation}
Plugging this estimate into \eqref{mainineqconvex}, we obtain

\begin{equation}\label{mainineqconvex2}
f(\hat{x}) \geqslant \frac{1}{2d}\sli_{j=1}^{2d}f(\tilde{x}_j) - \frac{\eta^2 N^{-2/d}\cdot 2s (s+3)(1-4\eta/c_s)^{-s-2}} {|\tilde y-\hat{x}|^{s+2}}.
 \end{equation}

We proceed with the Taylor formula for $f(\tilde x_j)$. We first write it for $j=1$. Recall that $\tilde{x}_1 = \hat{x}-\ep e_1$.  Since $|e_1|=1$, we get for some $\xi \in B(\hat{x}, |\tilde x_1-\hat{x}|) = B(\hat{x}, \ep)$,
\begin{multline}\label{supertaylorconvex}
f(\tilde x_1)=f(\hat{x}-\ep e_1) \\= f(\hat{x})-s\ep \frac{(\tilde y-\hat{x})\cdot e_1}{|\tilde y-\hat{x}|^{s+2}} +
\frac{\ep^2}2 \cdot \left(-s\cdot \frac{e_1\cdot e_1}{|\tilde y-\hat{x}|^{s+2}}+s(s+2)\frac{((\tilde y-\hat{x})\cdot e_1)^2}{|\tilde y-\hat{x}|^{s+4}}\right) \\
\hspace{2cm}+\frac{\ep^3}6 \cdot \left(-3s(s+2)\cdot \frac{((\tilde y-\xi)\cdot e_1)\cdot (e_1\cdot e_1)}{|\tilde y-\xi|^{s+4}} + s(s+2)(s+4)\cdot \frac{((\tilde y-\xi)\cdot e_1)^3}{|\tilde y-\xi|^{s+6}}\right)  \\
=f(\hat{x})-s\ep \frac{(\tilde y-\hat{x})\cdot e_1}{|\tilde y-\hat{x}|^{s+2}} +
\frac{\ep^2}2 \cdot \left(-s\cdot \frac{1}{|\tilde y-\hat{x}|^{s+2}}+s(s+2)\frac{((\tilde y-\hat{x})\cdot e_1)^2}{|\tilde y-\hat{x}|^{s+4}}\right)  \\
+\frac{\ep^3}6 \cdot \left(-3s(s+2)\cdot \frac{((\tilde y-\xi)\cdot e_1)}{|\tilde y-\xi|^{s+4}} + s(s+2)(s+4)\cdot \frac{((\tilde y-\xi)\cdot e_1)^3}{|\tilde y-\xi|^{s+6}}\right)
.
\end{multline}

Next we estimate the remainder term involving $\xi$. As before,
$$
|\tilde y-\xi|\geqslant |\tilde y - \hat{x}| - |\xi-\hat{x}|\geqslant |\tilde y - \hat{x}|- \tau N^{-1/d}\geqslant (1-2\tau/c_s)|\tilde y-\hat{x}|.
$$
This implies
\begin{multline}\label{errorconvex}
\left|-3s(s+2)\cdot \frac{(\tilde y-\xi)\cdot e_1}{|\tilde y-\xi|^{s+4}} + s(s+2)(s+4)\cdot \frac{((\tilde y-\xi)\cdot e_1)^3}{|\tilde y-\xi|^{s+6}}\right| \\
 \leqslant s(s+2)(s+7) \cdot (1-2\tau/c_s)^{-s-3}\cdot \frac{1}{|\tilde y-\hat{x}|^{s+3}}.
\end{multline}

Using the formula \eqref{supertaylorconvex} with $\tilde{x}_1$ replaced by $\tilde{x}_j$ we obtain an equation for $f(\tilde{x}_j)$ which, when substituted along with \eqref{errorconvex} into \eqref{mainineqconvex2}, yields
\begin{multline}\label{tirlim}
\frac{\ep^2}2 \left(-s\frac{1}{|\tilde{y}-\hat{x}|^{s+2}}+\frac{s(s+2)}{d}\cdot \frac{1}{|\tilde y-\hat{x}|^{s+2}}\right) \\
-\frac{\ep^3}6 s(s+2)(s+7) \cdot (1-2\tau/c_s)^{-s-3}\cdot \frac{1}{|\tilde y-\hat{x}|^{s+3}}\\ 
-\eta^2 N^{-2/d}\cdot 4s (s+3)(1-4\eta/c_s)^{-s-2}\cdot \frac{1}{|\tilde y-\hat{x}|^{s+2}}\leqslant 0.
\end{multline}
We remark that the first term in \eqref{tirlim} is, up to a constant factor, the Laplacian, in $x$, of the function $f(x)$. Although $f(x)$ is neither convex nor concave (for some choices of $\tilde{y}$, about which we have no information), the Laplacian $\Delta f(x)$ is always positive, which plays an essential role in our argument. Indeed, the need for at least $2d$ points $\{x_j\}_{j=1}^{2d}$ enables the definition of $\{\tilde{x}_j\}_{j=1}^{2d}$ so that the leading terms in the Taylor formula vanish leaving the positive second term. This will enable us to arrive at a contradiction to \eqref{tirlim} as we now explain. 

Recalling from \eqref{trumpumpum} that $|\tilde{y}-\hat{x}|\geqslant (c_s/2)\cdot N^{-1/d}$, we multiply \eqref{tirlim} by $2|\tilde{y}-\hat{x}|^{s+2}$ and divide by $sN^{-2/d}$ to obtain
\begin{multline}\label{contr321}
\frac{s+2-d}d\tau^2-2/3\tau^3 N^{-1/d}\cdot (s+2)(s+7)(1-2\tau/c_s)^{-s-3}c_s^{-1}\\
- 8\eta^2 (s+3)(1-4\eta/c_s)^{-s-2} \leqslant 0.
\end{multline}
Since $s>d$, this is impossible if $\tau$ is a suitable large multiple (depending on $s$) of $\eta$ and $\eta$ is small, and so the first assertion of Theorem \ref{thsepar} holds in this case. Observe that \eqref{contr321} fails if $\eta=\eta_s=c_s/s$ and $s$ is sufficiently large. Hence from Proposition \ref{babyhole} the family $\Omega_s$ is weakly well-separated with $M=2d-1$ and parameter $\eta_s$ with $\lim_{s\to \infty} \eta_s^{1/s}=1$.

\bigskip

\paragraph{\bf{Case 2:}} $\textup{dist}(\hat{x}, \partial A)<\ep$. Without loss of generality, we assume $\bar{x}+\ep e_d\not\in A$. We again take the point $x_\partial \in \partial A$ that achieves this distance and argue as in Lemma \ref{vovsestoroni}. We see that for any $j\leqslant 2d-2$ the points $\tilde{x}_j$, defined as above, lie in the set $A$. We redefine
$$
\tilde{x}_{2d-1}:=\tilde{x}_{2d}:=\hat{x} - \ep e_d,
$$
and let $\tilde{y}$ be as in \eqref{ytildedefin}.
The Taylor expansions of the terms on the left in \eqref{mainineqconvex} yield the following analog of \eqref{tirlim}:
\begin{multline}\label{blahblah123}
-s\frac{\ep}d \cdot \frac{\tilde y(d)-\hat{x}(d)}{|\tilde y - \hat{x}|^{s+2}}+\frac{\ep^2}2 \left(-s\frac{1}{|\tilde{y}-\hat{x}|^{s+2}}+\frac{s(s+2)}{d}\cdot \frac{1}{|\tilde y-\hat{x}|^{s+2}}\right) \\
-\frac{\ep^3}6 s(s+2)(s+7) \cdot (1-2\tau/c_s)^{-s-3}\cdot \frac{1}{|\tilde y-\hat{x}|^{s+3}}  \\ 
-\eta^2 N^{-2/d}\cdot 4s (s+3)(1-4\eta/c_s)^{-s-2}\cdot \frac{1}{|\tilde y-\hat{x}|^{s+2}}\leqslant 0,
\end{multline}
and, consequently, we have the following analog of \eqref{contr321},
\begin{multline}\label{contragain123}
-2\tau \frac{N^{1/d}}d (\tilde{y}(d)-\hat{x}(d))+ \frac{s+2-d}{d}\tau^2- \\ -2/3\tau^3 N^{-1/d}\cdot (s+2)(s+7)(1-2\tau/c_s)^{-s-3}c_s^{-1}- \\ 8\eta^2 (s+3)(1-4\eta/c_s)^{-s-2} \leqslant 0.
\end{multline}

Since $\tilde{y}(d)\leqslant x_\partial(d)$ and $\hat{x}(d)>x_\partial(d)-\tau N^{-1/d}$, we obtain
$$
-2\tau \frac{N^{1/d}}d (\tilde{y}(d)-\hat{x}(d))+ \frac{s+2-d}{d}\tau^2 \geqslant \frac{s-d}{d}\tau^2;
$$
therefore, \eqref{contragain123} is impossible for suitably small choices of $\eta$ and $\tau$, which as in the Case 1 yields the assertion of Theorem \ref{thsepar}.
\end{proof}

\section{Weak separation for the cube}\label{sectioncube}
In this section we show how to modify the proof of Theorem \ref{thsepar} to a case when the boundary $\partial A$ is not smooth. Namely, we prove the weak well-separation result for the unit cube, Theorem \ref{thseparcube}.

We begin with the following lemma.
\begin{lemma}\label{noangle}
If $s>d$, $\omega^*_N$ is optimal for $\PP_s([0,1]^d; N)$, and $x\in \omega^*_N$, then there exists a constant $b_s$ that does not depend on $N$ such that
$$
\max_{j=1, \ldots, d}x(j)\geqslant b_s N^{-1/d}.
$$
Moreover, one can choose $b_s$ so that $\lim_{s\to \infty} b_s^{1/s}=1$. 
\end{lemma}
\begin{proof}
We proceed as in Lemma \ref{notnearboundary}. Denote $v:=(1, \ldots, 1)$ and $\tilde{x}:=x+\ep v$. If for some small number $c$ we have $\max_{j=1,\ldots, d}x(j) \leqslant cN^{-1/d}$, then $\tilde{x}\in [0,1]^d$. Further, set $\widetilde{\omega}_N:=(\omega^*_N\setminus \{x\})\cup \{\tilde{x}\}$. If $\tilde{y}$ minimizes $P_s([0,1]^d, \widetilde\omega_N)$, then we have
$$
|\tilde{y}-x|\leqslant |\tilde{y}-\tilde{x}|,
$$
which implies
$$
(\tilde{y}-x)\cdot v \leqslant d\ep.
$$
Utilizing the definition of $v$ and taking $\ep\leqslant cN^{-1/d}$, we obtain
$$
\tilde y(j)\leqslant \sli_{j=1}^d \tilde y(j) \leqslant \sli_{j=1}^d x(j) + d\ep \leqslant d(cN^{-1/d}+\ep) \leqslant 2dcN^{-1/d}.
$$
Therefore,
$$
|\tilde{y}-\tilde{x}|\leqslant \sqrt{d}(\max_{j=1,\ldots, d}\tilde{y}(j)+\max_{j=1,\ldots, d}\tilde{x}(j)) \leqslant 4d\sqrt{d}\cdot cN^{-1/d}.
$$
If $c$ is small enough, this contradicts Proposition \ref{babyhole}.
\end{proof}

We are ready to prove Theorem \ref{thseparcube}.
\begin{proof}[Weak separation for the cube]
We again argue by contradiction. Suppose for $\eta>0$ and an optimal Riesz $s$-polarization configuration $\omega^*_N=\{x_1, \ldots, x_N\}$ we have $\{x_1, \ldots, x_{2d}\}\subset B(z, \eta N^{-1/d})$. Define
$$
\hat{x}:=\frac{x_1+\cdots+x_{2d}}{2d}\in [0,1]^d.
$$
Since $\hat{x}\in B(z, \eta N^{-1/d})$, we have $|x_j-\hat{x}|\leqslant 2\eta N^{-1/d}$ for every $j=1,\ldots, 2d$. 

Consider a small number $\tau>\eta$. We will choose it later to be a multiple of $\eta$. Set $\ep:=\tau N^{-1/d}$. We consider two cases.
\paragraph{\textbf{Case 1:}} $\textup{dist}(\hat{x}, \partial [0,1]^d)\geqslant \ep$.
In this case we proceed exactly as in the first case of Section \ref{sectionconvex} and get the same contradiction.

\bigskip

\paragraph{\textbf{Case 2:}}$\textup{dist}(\hat{x}, \partial [0,1]^d)<\ep$. We notice that since $|\hat{x}-x_j|<2\eta N^{-1/d}$, Lemma \ref{noangle} implies that $\hat{x}$ cannot be close to any vertex of the cube. Therefore, there exists at least one number $j$ such that $\hat{x}\pm \ep e_j \in [0,1]^d$. Without loss of generality, $j=1$. We now assume that for some $j_0=1,\ldots, N$ we have $\hat{x}\pm \ep e_j \in [0,1]^d$ for $j\leqslant j_0$, and $\hat{x}-\ep e_j \not\in [0,1]^d$ for $j>j_0$. Cases when $\hat{x}+\ep e_j\not\in [0,1]^d$ are treated similarly.
We define
\begin{alignat*}{2}
\tilde{x}_1&:=\hat{x}-\ep e_1,  & \tilde{x}_2&:=\hat{x}+\ep e_1,\\
& \qquad\qquad\qquad\qquad\ldots & \\ 
\tilde{x}_{2j_0-1}&:=\hat{x}-\ep e_{j_0}, & \tilde{x}_{2j_0}&:=\hat{x}+\ep e_{j_0},\\
\end{alignat*}
$\tilde{x}_{k}:=\hat{x}+\ep e_{\lfloor (k+1)/2 \rfloor}$ for $k=2j_0+1,\ldots, 2d$, and $\widetilde{\omega}_N:=\{\tilde{x}_1, \ldots, \tilde{x}_N\}$, where $\tilde{x}_j:=x_j$ for $j>2d$. Let $\tilde{y}$ such that 
$$
P_s(A; \widetilde\omega_N)=\sli_{j=1}^N\frac{1}{|\tilde{y}-\tilde{x}_j|^s}.
$$
Similarly to \eqref{blahblah123}, we get
\begin{multline}\label{blahblah321}
s\frac{\ep}d \cdot \sli_{j>j_0} \frac{\tilde y(j)-\hat{x}(j)}{|\tilde y - \hat{x}|^{s+2}}+\frac{\ep^2}2 \left(-s\frac{1}{|\tilde{y}-\hat{x}|^{s+2}}+\frac{s(s+2)}{d}\cdot \frac{1}{|\tilde y-\hat{x}|^{s+2}}\right) \\
-\frac{\ep^3}6 s(s+2)(s+7) \cdot (1-2\tau/c_s)^{-s-3}\cdot \frac{1}{|\tilde y-\hat{x}|^{s+3}} \\ 
-\eta^2 N^{-2/d}\cdot 4s (s+3)(1-4\eta/c_s)^{-s-2}\cdot \frac{1}{|\tilde y-\hat{x}|^{s+2}}\leqslant 0,
\end{multline}

Notice that if $\tilde{y}(j)\geqslant \hat{x}(j)$, then 
$$
s\frac{\ep}d \cdot  \frac{\tilde y(j)-\hat{x}(j)}{|\tilde y - \hat{x}|^{s+2}} \geqslant 0.
$$
If $\tilde{y}(j)<\hat{x}(j)$, then we estimate $\tilde{y}(j)-\hat{x}(j)\geqslant -\hat{x}(j) \geqslant -\ep$.  Since $j_0>1$, we have at most $d-1$ numbers $j$ with $j>j_0$. Therefore, \eqref{blahblah321} implies
\begin{multline}
\ep^2 \cdot \frac{s+2-d-2(d-1)}{2d} \cdot \frac{1}{|\tilde{y}-\hat{x}|^{s+2}}  \\
- \frac{\ep^3}6 s(s+2)(s+7) \cdot (1-2\tau/c_s)^{-s-3}\cdot \frac{1}{|\tilde y-\hat{x}|^{s+3}}  \\ 
-\eta^2 N^{-2/d}\cdot 4s (s+3)(1-4\eta/c_s)^{-s-2}\cdot \frac{1}{|\tilde y-\hat{x}|^{s+2}}\leqslant 0,
\end{multline}
which for suitably chosen $\eta$ and $\tau$ gives a contradiction if $s>3d-4$. As with Theorem \ref{thsepar}, it follows that $\eta=\eta_s$ can be taken so that $\lim_{s\to\infty}\eta_s^{1/s}=1$.

\end{proof}

\section{Weak separation on the sphere and spherical caps}\label{sectionsphere}
In this section we prove Theorem \ref{strconv} when $A=\S^d$ or when $A\subset \S^d$ is a spherical cap. We proceed as in Section \ref{sectionconvex}. However, computations will be different since the sphere $\S^d$ is not ``flat''. We start with the following result.
\begin{theorem}[Weak separation on the sphere] \label{weakseparsphere}
Consider the unit sphere $\S^d\subset \R^{d+1}$, and $s>d$ of $s\in [d-1,d)$. Then there exists a number $\eta>0$ such that for any $N$, any optimal configuration $\omega^*_N$ and any point $z\in \R^{d+1}$, we have
$$
\#(\omega_N \cap B(z, \eta N^{-1/d}))\leqslant 2d-1.
$$
Moreover, for large values of $s$ we can choose $\eta=\eta_s$ with
$$
\lim_{s\to \infty} \eta_s^{1/s} =1.
$$
\end{theorem}
\begin{proof}
Assume the theorem is false: there exists a ball $B(z, \eta N^{-1/d})$ and an optimal configuration $\omega^*_N=\{x_1, \ldots, x_N\}$ such that $\{x_1, \ldots, x_{2d}\}\subset  B(z, \eta N^{-1/d})$. 

Without loss of generality, we can assume $z'=0$ and $z(d+1)>0$. Denote
$$
\hat{x}':=\frac{x_1'+\cdots+x_{2d}'}{2d},
$$
and 
$$
\hat{x}(d+1):=\sqrt{1-|\hat{x}'|^2}.
$$
Since $|x_j-z|<\eta N^{-1/d}$ for $j=1, \ldots, 2d$, then $|x_j'|=|x_j'-z'|<\eta N^{-1/d}$; thus
$$
|\hat{x}'|<\eta N^{-1/d},
$$
and 
$$
1-\eta^2 N^{-2/d}\leqslant \hat{x}(d+1)\leqslant 1, \;\;\;\;\; 1-\eta^2 N^{-2/d}\leqslant x_j(d+1)\leqslant 1.
$$
Therefore, 
$$
-\eta^2 N^{-2/d}\leqslant x_j(d+1)-\hat{x}(d+1)\leqslant \eta^2 N^{-2/d},
$$
which implies for $\eta$ sufficiently small
$$
|x_j-\hat{x}|^2 = |x_j'-\hat{x}'|^2 + (x_j(d+1)-\hat{x}(d+1))^2 \leqslant 4\eta^2 N^{-2/d}+ \eta^4 N^{-4/d}\leqslant 5\eta^2 N^{-2/d}.
$$
We conclude that
$$
\{x_1, \ldots, x_{2d}\}\subset B(\hat{x}, \sqrt{5}\eta N^{-1/d}).
$$
Since the problem is rotation-invariant, we can assume $\hat{x}=e_{d+1}=(0, 0, \ldots, 0, 1)$ --- the North pole of the sphere. 

Fix a small number $\tau$, with $\eta < \tau < c_s/20$. We will choose $\tau$ at the end of the proof. Set
$$
\ep:=\tau N^{-1/d}.
$$

Note that $\{e'_1, \ldots, e'_d\}$ is the canonical orthonormal basis in $\R^d$; denote
\begin{alignat*}{2}
v_1&:=e_1, &  v_2&:=-e_1, \\
v_3&:=e_2, & v_4&:=-e_2, \\
& \qquad \qquad \qquad \ldots & \\
v_{2d-1}&:=e_d, & v_{2d}&:=-e_d.
\end{alignat*}
For $j=1, \ldots, 2d$ set 
$$
\tilde{x}_j':=\hat{x}'+\ep v_j=\ep v_j, \;\;\;\;\; \tilde{x}_j(d+1):=\sqrt{1-|\tilde{x}_j'|^2},
$$ 
and $\tilde{x}_j:=x_j$ if $j>2d$. For $\widetilde{\omega}_N:=\{\tilde{x}_1, \ldots, \tilde{x}_N\}$ let $\tilde{y}$ be such that 
$$
P_s(\mathbb{S}^d, \widetilde{\omega}_N) = \sli_{j=1}^N \frac{1}{|\tilde{y}-\tilde{x}_j|^s}.
$$
As before, denote
$$
f(x):=\frac{1}{|\tilde{y}-x|^s}.
$$ 
Estimates
\begin{equation}\label{mainineq}
\sli_{j=1}^N \frac{1}{|\tilde y-x_j|^s}\geqslant \inf_{y\in \S^d}\sli_{j=1}^N \frac{1}{|y-x_j|^s} =\PP_s(\S^d;N)\geqslant P_s(\S^d; \widetilde{\omega}_N)= \sli_{j=1}^N \frac{1}{|\tilde{y}-\tilde{x}_j|^s},
\end{equation}
imply, after utilizing that $x_j=\tilde{x}_j$ for $j\geqslant 2d+1$, that
\begin{equation}
\sli_{j=1}^{2d} \frac{1}{|\tilde y-x_j|^s}\geqslant \sli_{j=1}^{2d} \frac{1}{|\tilde{y}-\tilde{x}_j|^s}.
\end{equation}

Then from Taylor formula about $\hat{x}$ we have for $x\in\{x_1,\ldots, x_{2d}\}$ for some $\xi =\xi(x)\in B(\hat{x}, |x-\hat{x}|)$,
$$
f(x)=f(\hat{x}) + s\frac{(y-\hat{x})\cdot (x-\hat{x})}{|y-\hat{x}|^{s+2}} + \left(-s\cdot \frac{|x-\hat{x}|^2}{|y-\xi|^{s+2}}+s(s+2)\cdot \frac{((y-\xi)\cdot (x-\hat{x}))^2}{|y-\xi|^{s+4}}\right).
$$
Recall that if $x=x_j$, $1\leqslant j\leqslant 2d$, then $|x-\hat{x}|\leqslant \sqrt{5}\eta N^{-1/d}$. Moreover, we know from Lemma \ref{babyhole} that $|\tilde{y}-\tilde{x}_j|\geqslant c_s N^{-1/d}$. This implies 
$$
|\tilde{y}-\hat{x}|=|\tilde{y}-\tilde{x}_1+\ep e_1|\geqslant (c_s-\tau)N^{-1/d}\geqslant (c_s/2)\cdot  N^{-1/d},
$$
and
$$
|\tilde y-\xi|\geqslant |\tilde y-\hat{x}|-|\hat{x}-\xi|\geqslant |\tilde y - \hat{x}|-|x-\hat{x}|\geqslant |\tilde y - \hat{x}|-\sqrt{5}\eta N^{-1/d} \geqslant (1-2\sqrt{5}\eta/c_s)|\tilde y-\hat{x}|.
$$
Therefore, for every $j=1,\ldots, 2d$ we have
$$
f(x_j)\leqslant f(\hat{x})+s\frac{(\tilde y-\hat{x})\cdot(x_j-\hat{x})}{|\tilde y-\hat{x}|^{s+2}} + 5s(s+3)\eta^2 N^{-2/d}(1-2\sqrt{5}\eta/c_s)^{-s-2}\cdot \frac{1}{|\tilde y-\hat{x}|^{s+2}}.
$$
Summing these inequalities over $j$ and recalling that $x_1'+\cdots+x_{2d}'=(2d)\cdot e'=0$, we obtain
\begin{multline}
\sli_{j=1}^{2d}\frac{1}{|\tilde y-x_j|^s} \leqslant 2d\cdot f(\hat{x}) + s \frac{(\tilde y(d+1)-1)\cdot (x_{1}(d+1)+\cdots+x_{2d}(d+1)-2d)}{|\tilde y-\hat{x}|^{s+2}} \\
+ 10sd (s+3)\cdot \eta^2 N^{-2/d}\cdot(1-2\sqrt{5}\eta/c_s)^{-s-2}\cdot \frac{1}{|\tilde y-\hat{x}|^{s+2}}.
\end{multline}

From $|\tilde y(d+1)-1|\leqslant 2$ and $|x_{j}(d+1)-1| = 1-x_j(d+1)\leqslant \eta^2 N^{-2/d}$, we get
$$
\sli_{j=1}^{2d}\frac{1}{|\tilde y-x_j|^s} \leqslant 2d\cdot f(\hat{x}) + \eta^2 N^{-2/d}\cdot \Big(4sd+10sd (s+3)(1-2\sqrt{5}\eta)^{-s-2}\Big)\cdot \frac{1}{|\tilde y-\hat{x}|^{s+2}}.
$$

Plugging this estimate in \eqref{mainineq}, we obtain
\begin{equation}\label{mainineq2}
f(\hat{x}) \geqslant \frac{1}{2d}\sli_{j=1}^{2d}f(\tilde{x}_j) - \eta^2 N^{-2/d}\cdot \Big(2s+5s (s+3)(1-2\sqrt{5}\eta)^{-s-2}\Big)\cdot \frac{1}{|\tilde y-\hat{x}|^{s+2}}.
\end{equation}
We proceed with the Taylor formula for $f(\tilde x_j)$ about $\hat{x}$. We first write it for $j=1$. Recall that $\tilde{x}_1 = (\ep e'_1, \sqrt{1-\ep^2})$. Setting $v:=\tilde{x}_1-\hat{x} = (\ep e'_1, \sqrt{1-\ep^2}-1)$, we obtain for some \\ $\xi \in B(\hat{x}, |\tilde x_1-\hat{x}|) \subset B(\hat{x}, \sqrt{2}\ep)$,
\begin{multline}\label{supertaylor}
f(\tilde x_1)=f(\hat{x}+v) \\= f(\hat{x})+s\frac{(\tilde y-\hat{x})\cdot v}{|\tilde y-\hat{x}|^{s+2}} +
\frac12 \cdot \left(-s\cdot \frac{v\cdot v}{|\tilde y-\hat{x}|^{s+2}}+s(s+2)\frac{((\tilde y-\hat{x})\cdot v)^2}{|\tilde y-\hat{x}|^{s+4}}\right)  \\
+ \frac16 \cdot \left(-3s(s+2)\cdot \frac{((\tilde y-\xi)\cdot v)\cdot (v\cdot v)}{|\tilde y-\xi|^{s+4}} + s(s+2)(s+4)\cdot \frac{((\tilde y-\xi)\cdot v)^3}{|\tilde y-\xi|^{s+6}}\right).
\end{multline}
 
We first estimate the remainder term involving $\xi$. As before,
$$
|\tilde y-\xi|\geqslant |\tilde y - \hat{x}| - |\xi-\hat{x}|\geqslant |\tilde y - \hat{x}|- \sqrt{2}\tau N^{-1/d}\geqslant (1-2\sqrt{2}\tau/c_s)|\tilde y-\hat{x}|.
$$
Thus,
\begin{multline}\label{error}
\left|-3s(s+2)\cdot \frac{((\tilde y-\xi)\cdot v)\cdot (v\cdot v)}{|\tilde y-\xi|^{s+4}} + s(s+2)(s+4)\cdot \frac{((\tilde y-\xi)\cdot v)^3}{|\tilde y-\xi|^{s+6}}\right| \\
\leqslant
 s(s+2)(s+7)\cdot |v|^3 \cdot (1-2\sqrt{2}\tau/c_s)^{-s-3}\cdot \frac{1}{|\tilde y-\hat{x}|^{s+3}} \\
\leqslant 
2\sqrt{2}s(s+2)(s+7)\ep^3 \cdot (1-2\sqrt{2}\tau/c_s)^{-s-3}\cdot \frac{1}{|\tilde y-\hat{x}|^{s+3}}.
\end{multline}

For every $j=1,\ldots, 2d$ write the Taylor formula similar to \eqref{supertaylor}; in view of the estimate \eqref{error}, we get from\eqref{mainineq2},
\begin{multline}
s\cdot \frac{(\tilde y(d+1)-1)(\sqrt{1-\ep^2}-1)}{|\tilde y-\hat{x}|^{s+2}} 
\\ +\frac12 \left(-s\frac{2-2\sqrt{1-\ep^2}}{|\tilde{y}-\hat{x}|^{s+2}}+\frac{s(s+2)}{2d}\cdot \frac{2\ep^2 |\tilde y'|^2 + 2d (\tilde y(d+1)-1)^2 (\sqrt{1-\ep^2}-1)^2}{|\tilde y-\hat{x}|^{s+4}}\right)  \\
-2\sqrt{2}s(s+2)(s+7)\ep^3 \cdot (1-2\sqrt{2}\tau/c_s)^{-s-3}\cdot \frac{1}{|\tilde y-\hat{x}|^{s+3}} \\ 
-\eta^2 N^{-2/d}\cdot \Big(2s+5s (s+3)(1-2\sqrt{5}\eta / c_s)^{-s-2}\Big)\cdot \frac{1}{|\tilde y-\hat{x}|^{s+2}} \leqslant 0.
\end{multline}
Using
$$
2d (\tilde y(d+1)-1)^2 (\sqrt{1-\ep^2}-1)^2\geqslant 0,
$$
dividing by $s$ and multiplying by $|\tilde y-e|^{s+4}$, we obtain
\begin{multline}\label{mainineq33}
(\tilde y(d+1)-1)(\sqrt{1-\ep^2}-1)|\tilde y-\hat{x}|^{2} \\ 
+\frac12 \left(-(2-2\sqrt{1-\ep^2})|\tilde{y}-\hat{x}|^2+\frac{s+2}{2d}\cdot 2\ep^2 |\tilde y'|^2\right)  \\
-2\sqrt{2}(s+2)(s+7)\ep^3 \cdot (1-2\sqrt{2}\tau/c_s)^{-s-3}\cdot |\tilde y-\hat{x}|  \\
-\eta^2 N^{-2/d}\cdot \Big(2+5 (s+3)(1-2\sqrt{5}\eta / c_s)^{-s-2}\Big)\cdot |\tilde y-\hat{x}|^2 \leqslant 0
\end{multline}

Let us simplify first two terms. Notice that $|\tilde y-\hat{x}|^2 = 2-2\tilde y \cdot \hat{x} = 2-2\tilde y(d+1)$. We have:
\begin{multline}
(\tilde y(d+1)-1)(\sqrt{1-\ep^2}-1)|\tilde y-\hat{x}|^{2} \\ 
+\frac12 \left(-(2-2\sqrt{1-\ep^2})|\tilde{y}-\hat{x}|^2+\frac{s+2}{2d}\cdot 2\ep^2 |\tilde y'|^2\right) \\ 
=\tilde{y}(d+1)(\sqrt{1-\ep^2}-1)(2-2\tilde{y}(d+1))+\frac{s+2}{2d}(1-\tilde y(d+1)^2)\ep^2  \\
=|\tilde y-\hat{x}|^2 \cdot \Big((\sqrt{1-\ep^2}-1)\tilde y(d+1)+\ep^2 \frac{s+2}{4d}(1+\tilde y(d+1))\Big).
\end{multline}
If $\tilde{y}(d+1)<0$, we use that $\sqrt{1-\ep^2} -1 \leqslant -\frac{\ep^2}2$ to get
\begin{multline}\label{esttt1}
(\sqrt{1-\ep^2}-1)\tilde y(d+1)+\ep^2 \frac{s+2}{4d}(1+\tilde y(d+1))  \\ 
\geqslant\frac{\ep^2}{2}\left(-\tilde y(d+1)+\frac{s+2}{2d}(1+\tilde y(d+1))\right) 
\geqslant \frac{\ep^2}2 \cdot \min\left(\frac{s+2}{2d}, 1\right).
\end{multline}
If $\tilde{y}(d+1)\geqslant 0$, we use $\sqrt{1-\ep^2} -1 \geqslant -\frac{\ep^2}2-\frac{\ep^4}8$ to get
\begin{multline}\label{esttt2}
(\sqrt{1-\ep^2}-1)\tilde y(d+1)+\ep^2 \frac{s+2}{4d}(1+\tilde y(d+1)) \\ 
\geqslant \frac{\ep^2}{2}\left(-\tilde y(d+1)+\frac{s+2}{2d}(1+\tilde y(d+1))\right)  - \frac{\ep^4}8 \geqslant \frac{\ep^2}2 \min\left(\frac{s+2}{2d}, \frac{s+2-d}{d}\right) - \frac{\ep^4}8.
\end{multline}
Combining estimates \eqref{esttt1} and \eqref{esttt2}, we get
\begin{multline}
(\tilde y(d+1)-1)(\sqrt{1-\ep^2}-1)|\tilde y-\hat{x}|^{2} \\ 
+\frac12 \left(-(2-2\sqrt{1-\ep^2})|\tilde{y}-\hat{x}|^2+\frac{s+2}{2d}\cdot 2\ep^2 |\tilde y'|^2\right)   \\ 
\geqslant |\tilde y - \hat{x}|^2 \cdot \left(\ep^2\min\left(\frac12, \frac{s+2}{4d}, \frac{s+2-d}{2d}\right) - \frac{\ep^4}8\right).
\end{multline}
Plugging this estimate into \eqref{mainineq33} and dividing by $|\tilde y-\hat{x}|^2$, we obtain:
\begin{multline}
\ep^2\min\left(\frac12, \frac{s+2}{4d}, \frac{s+2-d}{2d}\right) - \frac{\ep^4}8  \\
-2\sqrt{2}(s+2)(s+7)\ep^3 \cdot (1-2\sqrt{2}\tau/c_s)^{-s-3}\cdot |\tilde y-\hat{x}|^{-1}  \\ 
-\eta^2 N^{-2/d}\cdot \Big(2+5 (s+3)(1-2\sqrt{5}\eta / c_s)^{-s-2}\Big)  \leqslant 0
\end{multline}
We now recall that $\ep=\tau N^{-1/d}$. Denote 
$$
C(s,d):=\min\left(\frac12, \frac{s+2}{4d}, \frac{s+2-d}{2d}\right).
$$
Then 
\begin{multline}\label{mainineqlast}
C(s,d)\tau^2 - \frac{\tau^4 N^{-2/d}}8  \\ 
-4d\sqrt{2}(s+2)(s+7)(1-2\sqrt{2}\tau/c_s)^{-s-3} \tau^3  \cdot (N^{-1/d}|\tilde y-\hat{x}|^{-1})  \\ 
-\eta^2 \cdot \Big(4d+10d (s+3)(1-2\sqrt{5}\eta / c_s)^{-s-2}\Big)  \leqslant 0.
\end{multline}
We should finally recall that $N^{-1/d}|\tilde y-\hat{x}|^{-1}\leqslant 2/c_s$. Thus, we can choose sufficiently small $\eta$ and $\tau$ such that the left-hand side of \eqref{mainineqlast} is strictly positive, which is a contradiction. Finally, as in Section \ref{sectionconvex}, for large values of $s$ we can choose $\eta=\eta_s$ with $\eta_s^{1/s} \to 1$ as $s\to \infty$.
\end{proof}

We proceed with the same statement for spherical caps $A\subset \S^d$. As in the case of bodies in $\R^d$, we will need to deal with the case when point $\hat{x}$ is near the boundary. 

\begin{corollary}[Weak separation on the caps] \label{weakseparspherecaps}
Consider the unit sphere $\S^d\subset \R^{d+1}$, and $s>d$. Let $A\subset \S^d$ be a spherical cap, $A=\{x\in \S^d\colon x(1)\geqslant t_0\}$. Then there exists a number $\eta>0$ such that for any $N$, any optimal configuration $\omega^*_N$ for $\PP_s(A; N)$, and any point $z\in \R^{d+1}$ we have
$$
\#(\omega_N \cap B(z, \eta N^{-1/d}))\leqslant 2d-1.
$$
Moreover, for large values of $s$ we can choose $\eta=\eta_s$ so that
$$
\lim_{s\to \infty} \eta_s^{1/s} =1.
$$
\end{corollary}
\begin{proof}
For the sake of simplicity, we prove this corollary for $d=2$. The case of general $d$ can be treated similarly. We also assume $t_0\geqslant 0$. The case $t_0<0$ is done through the same estimates.

We again argue by contradiction. Assume for some small $\eta>0$ there exists a ball $B(z, \eta N^{-1/2})$ and an extremal configuration $\omega^*_N=\{x_1, \ldots, x_N\}$ such that $\{x_1, \ldots, x_{4}\}\subset B(z, N^{-1/2})$. Set
$$
\hat{x}':=\frac{x_1'+\cdots+x_{4}'}{4},
$$
and 
$$
\hat{x}(3):=\sqrt{1-|\hat{x}'|^2}.
$$
Recall that $x\in A$ if and only if $x(1)\geqslant t_0$. Thus, we see that $\hat{x}'\in A$, and, as before,
$$
\{x_1, \ldots, x_{4}\}\subset B(\hat{x}, \sqrt{5}\eta N^{-1/2}).
$$

Since the problem is rotation invariant, we can assume $\hat{x}=(\hat{t}, 0, \sqrt{1-\hat{t}^2})$ for some $\hat{t}\geqslant t_0$.

We denote
$$
v_1:=(-\sqrt{1-\hat t^2}, 0, \hat t), \;\;\;\; v_2:=(0,1,0).
$$

Set $\ep:=\tau N^{-1/2}$ and consider
\begin{align*}
&\tilde{x}_1:=\Big(\ep \sqrt{1-\hat{t}^2}+\hat{t}\sqrt{1-\ep^2(1-\hat{t}^2)}, 0, -\ep \hat{t}+\sqrt{1-\hat{t}^2}\cdot\sqrt{1-\ep^2(1-\hat{t}^2)}\Big), \\
&\tilde{x}_2:=\Big(-\ep \sqrt{1-\hat{t}^2}+\hat{t}\sqrt{1-\ep^2(1-\hat{t}^2)}, 0, \ep \hat{t}+\sqrt{1-\hat{t}^2}\cdot\sqrt{1-\ep^2(1-\hat{t}^2)}\Big),\\
& \tilde{x}_3:=\Big(\sqrt{1-\ep^2}\hat{t}, \ep, \sqrt{1-\hat{t}^2}\cdot \sqrt{1-\ep^2}\Big), \\
&\tilde{x}_4:=\Big(\sqrt{1-\ep^2}\hat{t}, -\ep, \sqrt{1-\hat{t}^2}\cdot \sqrt{1-\ep^2}\Big).
\end{align*}
If $\tilde{x}_1, \ldots, \tilde{x}_4\in A$, then we get the same contradiction as for the sphere $\mathbb{S}^d$. Thus, the only case we need to consider is when one of these points is not in $A$. 

A direct computation shows that
$$
\tilde{x}_1-\hat{x} = \left(\ep \sqrt{1-\hat{t}^2}-\frac{\hat t(1-\hat{t}^2)}2\ep^2, 0, -\ep t - \frac{(1-\hat{t}^2)^{3/2}}{2}\ep^2\right) + O(\ep^3),
$$
$$
\tilde{x}_2-\hat{x} = \left(-\ep \sqrt{1-\hat{t}^2}-\frac{\hat t(1-\hat{t}^2)}2\ep^2, 0, \ep t - \frac{(1-\hat{t}^2)^{3/2}}{2}\ep^2\right) + O(\ep^3),
$$
$$
\tilde{x}_3-\hat{x}=\left(-\frac{\hat{t}}2\ep^2, \ep, -\frac{\sqrt{1-\hat t^2}}2\ep^2\right)+O(\ep^3),
$$
$$
\tilde{x}_4-\hat{x}=\left(-\frac{\hat{t}}2\ep^2, -\ep, -\frac{\sqrt{1-\hat t^2}}2\ep^2\right)+O(\ep^3).
$$
Thus, $\tilde{x}_1(1)$ and $\tilde{x}_3(1)$ are greater or equal than $t_0$, and if $\tilde{x}_2(1)<t_0$ or $\tilde{x}_4(1)<t_0$, then
\begin{equation}\label{capnearboundary}
\hat{t}-\ep \sqrt{1-\hat{t}^2}-\frac{\hat t(1-\hat{t}^2)}2\ep^2\leqslant t_0.
\end{equation}

If this is the case, we define the points $\tilde{x}_1, \ldots, \tilde{x}_4$ differently; namely,
\begin{align*}
&\tilde{x}_1:=\Big(\ep \sqrt{1-\hat{t}^2}+\hat{t}\sqrt{1-\ep^2(1-\hat{t}^2)}, 0, -\ep \hat{t}+\sqrt{1-\hat{t}^2}\cdot\sqrt{1-\ep^2(1-\hat{t}^2)}\Big), \\
&\tilde{x}_2:=\Big(\ep \sqrt{1-\hat{t}^2}+\hat{t}\sqrt{1-\ep^2(1-\hat{t}^2)}, 0, -\ep \hat{t}+\sqrt{1-\hat{t}^2}\cdot\sqrt{1-\ep^2(1-\hat{t}^2)}\Big),\\
& \tilde{x}_3:=\Big(\hat{t}, \ep, \sqrt{1-\hat{t}^2-\ep^2}\Big), \\
&\tilde{x}_4:=\Big(\hat{t}, -\ep, \sqrt{1-\hat{t}^2-\ep^2}\Big).
\end{align*}
We set $\tilde{x}_j:=x_j$ for $j>4$, $\widetilde{\omega}_N:=\{\tilde{x}_1, \ldots, \tilde{x}_N\}$ and write the same Taylor formulas as before. We get
\begin{equation}
f(\hat{x}) \geqslant \frac{1}{2d}\sli_{j=1}^{2d}f_{\tilde{y}}(\tilde{x}_j) - \eta^2 N^{-2/d}\cdot \Big(2s5s (s+3)(1-2\sqrt{5}\eta/c_s)^{-s-2}\Big)\cdot \frac{1}{|\tilde{y}-\hat{x}|^{s+2}}.
\end{equation}
Expanding $f(\tilde{x}_j)$ about $\hat{x}$ as before, we get
\begin{multline}\label{hessianhessian}
\ep^2\cdot \left(\frac{|\tilde{y}-\hat{x}|^2}{2}(2-\hat t - (s+2)/2)+s\right)  \\
+ 2\ep\left((\tilde{y}(1)-\hat t)\sqrt{1-\hat t^2}-(\tilde{y}(3)-\sqrt{1-\hat t^2})\hat{t}\right)  \\
+\ep^2\cdot \left((\tilde{y}(1)-\hat t)\hat{t}+(\tilde{y}(3)-\sqrt{1-\hat{t}^2})\sqrt{1-\hat t^2} - \frac{\tilde{y}(3)-\sqrt{1-\hat{t}^2}}{\sqrt{1-\hat{t}^2}}\right) 
\\ -4\eta^2 N^{-2/d}\cdot \Big(2s+5s (s+3)(1-2\sqrt{5}\eta/c_s)^{-s-2}\Big) - \; \mbox{remainder terms involving $\xi$} \; \leqslant 0,
\end{multline}
where the remainder terms are handled exactly as in \eqref{error}.

We proceed with showing that the third term can not be a large negative number. In fact,
\begin{equation}\label{againerrorrr}
(\tilde{y}(1)-\hat t)\hat{t}+(\tilde{y}(3)-\sqrt{1-\hat{t}^2})\sqrt{1-\hat t^2} - \frac{\tilde{y}(3)-\sqrt{1-\hat{t}^2}}{\sqrt{1-\hat{t}^2}}=\tilde{y}(1)\hat{t}-\frac{\hat{t}^2}{\sqrt{1-\hat{t}^2}}\tilde{y}(3).
\end{equation}
If $\tilde{y}(3)< 0$, we see that this expression is non-neagtive. Otherwise, plugging
$$
\tilde{y}(1)\geqslant t_0\geqslant \hat{t}-\ep \sqrt{1-\hat{t}^2}-\frac{\hat t(1-\hat{t}^2)}2\ep^2,
$$
and $\tilde{y}(3)\leqslant \sqrt{1-t_0^2}$ into \eqref{againerrorrr}, we obtain
$$
(\tilde{y}(1)-\hat t)\hat{t}+(\tilde{y}(3)-\sqrt{1-\hat{t}^2})\sqrt{1-\hat t^2} - \frac{\tilde{y}(3)-\sqrt{1-\hat{t}^2}}{\sqrt{1-\hat{t}^2}} \geqslant -c\ep
$$
for some non-negative constant $c$, which depends only on $t_0$. 
We finally show how to estimate the second term of \eqref{hessianhessian}. Without loss of generality, we can assume this term is negative, in particular, $\hat{t}\not = 0$. The equality
$$
(\tilde{y}(1)-\hat t)\sqrt{1-\hat t^2}-(\tilde{y}(3)-\sqrt{1-\hat t^2})\hat{t} = \tilde{y}(1)\sqrt{1-\hat{t}^2}-\tilde{y}(3)\hat{t}.
$$
yields
$$
|\tilde{y}-\hat{x}|^2 = 2-2\tilde{y}(1)\hat t-2\tilde{y}(3)\sqrt{1-\hat{t}^2} \leqslant 2-2\tilde{y}(1)/\hat{t} \leqslant 2-2t_0/\hat{t} \leqslant \ep\sqrt{1-\hat{t}^2}+\frac{\hat{t}(1-\hat{t}^2)}2\ep^2 \leqslant c\ep,
$$
where again $c$ is a positive constant which depends only on $t_0$. 
On the other hand, 
$$
\tilde{y}(1)\sqrt{1-\hat{t}^2}-\tilde{y}(3)\hat{t}\geqslant -\ep - c\ep^2.
$$
Thus, inequality \eqref{hessianhessian} implies
$$
\ep^2(c\ep(2-\hat{t}-(s+2)/2)+s)-2\ep^2 -c\ep^3 - \; \mbox{remainder terms} \; \leqslant 0,
$$
which is impossible since $s>2$. 
\end{proof}

\section{Proofs of covering results}\label{sectionmesh}

\begin{proof}[Proof of Theorem \ref{thcovering}]
Fix an integer $N$. Since $\tilde{A}$ is a $d$-regular compact set, there exists 
a finite family of sets $\{Q_\alpha\}_{\alpha}$ with the following properties:
\begin{enumerate}
\item $\tilde{A}=\cup_\alpha Q_\alpha$ and the interiors of the sets $Q_\alpha$ are disjoint; furthermore, $\mu(Q_\alpha)=0$ for every $\alpha$, where $\mu$ is the measure from Definition \ref{defregular};
\item There exists a positive constant $a_1$ that does not depend on $N$, and points $z_\alpha\in Q_\alpha$ such that $B(z_\alpha, a_1\eta N^{-1/d})\cap \tilde{A}\subset Q_\alpha \subset B(z_\alpha, \eta N^{-1/d})$. 
\end{enumerate}
For the construction of such sets see, e.g., \cite{Christ1990}. 
 Notice that since $Q_\alpha \subset B(z_\alpha, \eta N^{-1/d})$, we have $\#(Q_\alpha\cap \omega_N )\leqslant M$.

Let $\mathcal{A}$ denote the set of indices $\alpha$ such that $Q_\alpha\cap \omega_N\not=\emptyset$. Since every $Q_\alpha$ can contain no more than $M$ points from $\omega_N$, we deduce that number of such indices is at least as large as $N/M$. 

Hereafter we follow an argument in \cite{Hardin2012}.

Without loss of generality, we assume $\rho_A(\omega_N)\geqslant 5\eta N^{-1/d}$.
Let $y\in A$ be such that $\min_{x_k\in \omega_N} |y-x_k|=\rho_A(\omega_N)$. For every $x_j\in \omega_N$ let $\alpha_j=\alpha$ denote the index such that $x_j\in Q_\alpha$ for some $\alpha$. If $x\in Q_\alpha$, then
$$
|y-x|\leqslant |y-x_j|+|x_j-x|\leqslant |y-x_j|+2\eta N^{-1/d}\leqslant |y-x_j|+\frac25\rho_A(\omega_N)\leqslant \frac75|y-x_j|.
$$
Consequently, 
\begin{equation}\label{anotherequationnnnn}
|y-x_j|^{-s}\leqslant \left(\frac75\right)^s \cdot \min_{x\in Q_\alpha} |y-x|^{-s}. 
\end{equation}
Furthermore, 
$$
|y-x|\geqslant |y-x_j|-|x_j-x|\geqslant |y-x_j|-2\eta N^{-1/d}\geqslant |y-x_j|-\frac25 \rho_A(\omega_N)\geqslant \frac35 \rho_A(\omega_N),
$$
which implies
$$
A\cap B(y, (3/5) \rho_A(\omega_N))\subset A\setminus \bigcup_{\alpha\in\mathcal{A}}Q_\alpha.
$$

For each $x_j\in Q_\alpha$ we see from \eqref{anotherequationnnnn} that
$$
\frac{1}{|y-x_j|^s}\leqslant \left(\frac 75\right)^s\frac{1}{\mu(Q_\alpha)}\int_{Q_\alpha}\frac{\textup{d}\mu(x)}{|y-x|^s}.
$$
Since $B(z_\alpha, a_1\eta N^{-1/d})\cap\tilde{A}\subset Q_\alpha$, we have by the $d$-regularity condition that $\mu(Q_\alpha)\geqslant c_1\cdot \eta^d/N$, where the positive constant $c_1$ does not depend on $s$. This implies from assumption \eqref{pspsps} that
\begin{multline}\label{anothermultlineee}
p_s N^{s/d}\leqslant \sli_{x_j\in \omega_N}\frac{1}{|y-x_j|^s}\leqslant M\cdot \left(\frac 75\right)^s\sli_{\alpha\in \mathcal{A}} \frac{1}{\mu(Q_\alpha)}\ili_{Q_\alpha}\frac{\textup{d}\mu(x)}{|y-x|^s} \\ 
\leqslant c_1^{-1} M\cdot \left(\frac 75\right)^s \cdot \eta^{-d}\cdot N \ili_{A\setminus B(y, (3/5) \rho_A(\omega_N))}\frac{\textup{d}\mu(x)}{|y-x|^s}\\
\leqslant c_1^{-1}\cdot c_2\cdot \frac{s}{s-d}\cdot M\cdot \left(\frac 75\right)^s \cdot \eta^{-d} \cdot N\cdot \left((3/5) \rho_A(\omega_N)\right)^{d-s},
\end{multline}
where $c_2$ does not depend on $s$. This yields, for $C_d:=c_{1}^{-1}\cdot c_2$,
$$
\rho_A(\omega_N)^{s-d}\leqslant C_d \cdot \frac{s}{s-d}\cdot \left(\frac 75\right)^s \cdot \frac1{p_s}\cdot \eta^{-d} \cdot M \cdot N^{-\frac{s-d}d},
$$
which implies
$$
\rho_A(\omega_N)\leqslant \left(C_d \cdot \frac{s}{s-d}\right)^{\frac{1}{s-d}}\cdot \left(\frac 73\right)^{\frac{s}{s-d}} \cdot p_s^{-\frac{1}{s-d}} \cdot \eta^{-\frac{d}{s-d}} \cdot M^{\frac{1}{s-d}} \cdot N^{-1/d},
$$
as claimed.
\end{proof}

\begin{proof}[Proof of Corollary \ref{coveringclopen}]
First, we prove that for any $\omega_N$ that is extremal for $\PP_s(A; N)$, there exists a positive constant $p_s$ with
$$
\inf_{y\in A} \sli_{x_j\in \omega_N} \frac{1}{|y-x_j|^s} \geqslant p_s N^{s/d}.
$$
We prove it for strongly convex $A\subset \R^d$ or $A=[0,1]^d$. The case $A=\mathbb{S}^d$ is similar. First, notice that for any $z\in A$ we have $A\subset z+[-\diam(A), \diam(A)]^d=:Q$.
For a fixed $N$ and a fixed constant $a$, consider a maximal set $\mathcal{E}$ such that for any $x,y\in \mathcal{E}$ we have $|x-y|\geqslant aN^{-1/d}$. The maximality of $\mathcal{E}$ implies that 
$$
A\subset \bigcup_{x\in \mathcal{E}}B(x, aN^{-1/d});
$$
thus $\rho_A(\mathcal{E})\leqslant aN^{-1/d}$.

On the other hand, we see that the sets $B(x, (a/3)N^{-1/d})\cap Q$ are disjoint. Thus, 
$$
\H_d(Q)\geqslant c_1 \cdot a^d\cdot N^{-1}\cdot \#(\mathcal{E}),
$$
which implies
$$
\#(\mathcal{E})\leqslant c_2 a^{-d}N,
$$
where $c_1$ and $c_2$ are positive constants that depend on $d$.
We now choose $a$ such that $c_2 a^{-d}=1$. This implies that there exists an $N$-point set $\widetilde{\omega}_N$ such that
$$
A\subset \bigcup_{\tilde{x}_j\in\widetilde\omega_N} B(\tilde x_j, aN^{-1/d}),
$$
where the number $a$ depends only on $A$ and $d$. In particular, $\rho_A(\widetilde{\omega}_N)\leqslant aN^{-1/d}$.

Observe that
\begin{multline}\label{frombelowcovering}
\inf_{y\in A} \sli_{x_j\in \omega_N} \frac{1}{|y-x_j|^s} =\PP_s(A; N)\geqslant P_s(A; \widetilde\omega_N)=\inf_{y\in A}\max_{\tilde x_j\in \widetilde\omega_N}\frac{1}{|y-\tilde x_j|^s} \\ =\frac{1}{\max_{y\in A}\min_{\tilde x_j\in \widetilde\omega_N} |y-\tilde x_j|^s} = \rho_A(\widetilde \omega_N)^{-s} \geqslant a^{-s} N^{s/d}.
\end{multline}
Thus, we can apply Theorem \ref{thcovering} with $p_s=a^{-s}$ to obtain
$$
\rho_{A}(\omega_N)\leqslant R_s N^{-1/d}
$$
for 
$$
R_s=\left(\frac{C_d\cdot M \cdot s\cdot  7^s \cdot a^s}{(s-d)\cdot 5^s \cdot \eta_s^{d} }\right)^{\frac{1}{s-d}},
$$
where $\eta_s$ is the constant from Theorem \ref{thsepar} or Theorem \ref{thseparcube}.

To complete the proof, recall that we have $\lim_{s\to \infty}\eta_s^{1/s} =1$, therefore for large values of $s$ we have $R_s\leqslant R_0$ for some positive $R_0$.
\end{proof}

\section{Proof of Best Covering Results}\label{sectionlimit}

We begin by remarking that in Section \ref{sectionmesh} we have seen that if $A$ is $d$-regular, then for some positive constants $a$ and $b$ we have $aN^{-1/d}\leqslant \rho_A(N) \leqslant bN^{-1/d}$, where $\rho_A(N)$ is defined in \eqref{defbestcovering}.

\begin{proof}[Proof of Theorem \ref{bestcovering}]
Using the same argument as in \eqref{frombelowcovering}, we see that
$$
\PP_s(A; N)\geqslant \frac{1}{\rho_A(N)^{s}}.
$$
Therefore,
$$
\left(\lim_{N\to \infty} \frac{\PP_s(A; N)}{N^{s/d}}\right)^{1/s}\geqslant \frac{1}{\liminf_{N\to \infty}(N^{1/d}\rho_A(N))},
$$
which implies
\begin{equation}\label{liminfliminf}
\liminf_{s\to\infty}\left(\lim_{N\to \infty} \frac{\PP_s(A; N)}{N^{s/d}}\right)^{1/s}\geqslant \frac{1}{\liminf_{N\to \infty}(N^{1/d}\rho_A(N))}.
\end{equation}

On the other hand, for a fixed positive integer $N$ and large $s$ consider an $N$-point configuration $\omega^*_N=\{x_1, \ldots, x_N\}$ such that $
\PP_s(A; N)=P_s(A; \omega^*_N)$. Corollary \ref{coveringclopen} implies that if $s$ is large enough, then $\rho_A(\omega^*_N)\leqslant R_0 N^{-1/d}$, where $R_0$ depends neither on $N$, nor on $s$. We also recall that the Theorems \ref{thsepar} and \ref{thseparcube} imply that for any large value of $s$ there exists a number $\eta_s>0$ such that for any $z\in \R^d$ we have $\#(\omega^*_N\cap B(z, \eta_s N^{-1/d}))\leqslant 2d-1$ and $\lim_{s\to \infty}\eta_s^{1/s}=1$.  

We now take a point $y\in A$ such that
\begin{equation}\label{centerofhole}
\min_{j=1,\ldots, N}|y-x_j|=\rho_A(\omega^*_N),
\end{equation}
and set
$$ 
B_n:=B(y, \, n \rho_A(\omega^*_N)) \setminus B(y, \,(n-1) \rho_A(\omega^*_N)), 
$$
where $n$ is an integer with $n\geqslant 2$. Since the open ball $B(y, \rho_A(\omega^*_N))$ does not intersect $\omega^*_N$, we have
$$
\omega^*_N\subset \bigcup_{n=2}^\infty B_n.
$$
Notice that for any $n\geqslant 2$ we have $B_n\subset B(y, \, nR_0N^{-1/d})$; thus, there exists a constant $\tilde{C}_1$ that does not depend on $s$ such that the annulus $B_n$ can be covered by $\tilde C_1 R_0^d  n^d\eta_s^{-d}=:C_2 n^d \eta_s^{-d}$ balls of radius $\eta N^{-1/d}$. Thus, for any $n\geqslant 2$ we have
$$
\#(B_n\cap \omega^*_N)\leqslant  C_2 (2d-1)  n^d \eta_s^{-d} =:C_3  n^d  \eta_s^{-d}.
$$
For $y$ defined in \eqref{centerofhole} we have
$$
\PP_s(A; N)\leqslant \sli_{x\in \omega^*_N}\frac{1}{|y-x|^s}\leqslant \sli_{n=2}^\infty \left(\sli_{x\in \omega^*_N \cap B_n}\frac{1}{|y-x|^s}\right).
$$
By the definition of $B_n$, for any $x\in B_n$ we have $|y-x|\geqslant (n-1)\rho_A(\omega^*_N)$, which implies
\begin{equation}\label{aibolit}
\PP_s(A; N)\leqslant \sli_{n=2}^\infty C_3 n^d \eta_s^{-d} (n-1)^{-s}\rho_A(\omega^*_N)^{-s} = C_3 \eta_s^{-d}\rho_A(\omega^*_N)^{-s} \sli_{n=2}^\infty n^{d} (n-1)^{-s}.
\end{equation}
Dividing by $N^{s/d}$ and using that $\rho_A(\omega^*_N)\geqslant \rho_A(N)$, we obtain
\begin{equation}\label{trubadur}
\frac{\PP_s(A; N)}{N^{s/d}}\leqslant C_3 \eta_s^{-d} \sli_{n=1}^\infty n^{d-s} \cdot \left(\frac{1}{N^{1/d}\rho_A(N)}\right)^{s},
\end{equation}
which implies
\begin{equation}\label{petuh}
\left(\lim_{N\to \infty}\frac{\PP_s(A; N)}{N^{s/d}}\right)^{1/s}\leqslant C_3^{1/s} \eta_s^{-d/s} \left(\sli_{n=2}^\infty n^{d-s}\right)^{1/s} \cdot \frac{1}{\limsup_{N\to\infty}(N^{1/d}\rho_A(N))}.
\end{equation}
Taking $\limsup_{s\to \infty}$, we obtain
\begin{equation}\label{limsuplimsup}
\limsup_{s\to\infty}\left(\lim_{N\to \infty}\frac{\PP_s(A; N)}{N^{s/d}}\right)^{1/s}\leqslant \frac{1}{\limsup_{N\to\infty}(N^{1/d}\rho_A(N))}.
\end{equation}
Estimates \eqref{liminfliminf} and \eqref{limsuplimsup} imply that $\lim_{N\to \infty}N^{1/d}\rho_A(N)$ and $\lim_{s\to \infty}\left(\lim_{N\to \infty}\PP_s(A; N)N^{-s/d}\right)^{1/s}$ exist and satisfy
$$
\lim_{s\to\infty}\left(\lim_{N\to \infty}\frac{\PP_s(A; N)}{N^{s/d}}\right)^{1/s}= \frac{1}{\lim_{N\to\infty}(N^{1/d}\rho_A(N))}.
$$
\end{proof}
As an immediate consequence of Theorem \ref{bestcovering} we state the following corollary about behavior of covering radii of optimal $s$-Riesz polarization configurations as $s\to \infty$.
\begin{corollary}
Suppose $A$ is a $d$-admissible set or $A=[0,1]^d$. For every $N\geqslant 1$ and every $s>d$ fix an $N$-point configuration $\omega_N^s$ such that $\PP_s(A; N)=P_s(A; \omega^s_N)$. Then the following limits exist and satisfy
\begin{equation}\label{chichi}
\lim_{s\to\infty}\lim_{N\to\infty}N^{1/d}\rho_A(\omega_N^s) = \lim_{N\to \infty}N^{1/d}\rho_A(N).
\end{equation}
\end{corollary}
\begin{proof}
Arguing as in \eqref{frombelowcovering}, we get that
$$
\PP_s(A; N)\geqslant \frac{1}{\rho_A(\omega_N^s)^{s}},
$$
which implies from \eqref{princess} that 
$$
\lim_{N\to\infty}N^{1/d}\rho_A(N)=\lim_{s\to\infty}\left(\lim_{N\to \infty}\frac{\PP_s(A; N)}{N^{s/d}}\right)^{-1/s} \leqslant \liminf_{s\to\infty}\Big[\liminf_{N\to\infty}N^{1/d}\rho_A(\omega_N^s)\Big].
$$
On the other hand, arguing as in \eqref{aibolit}, \eqref{trubadur} and \eqref{petuh} we get
$$
\lim_{N\to\infty}N^{1/d}\rho_A(N)=\lim_{s\to\infty}\left(\lim_{N\to \infty}\frac{\PP_s(A; N)}{N^{s/d}}\right)^{-1/s}\geqslant \limsup_{s\to\infty}\Big[\limsup_{N\to\infty}(N^{1/d}\rho_A(\omega_N^s))\Big]	,
$$
and \eqref{chichi} follows.
\end{proof}

\section{Proof of Proposition \ref{babyhole}}\label{secappendix}

\begin{proof}[Proof of Proposition \ref{babyhole} for $s>d$]

Take a positive integer $N$, an $N$-point configuration $\omega_N$ and the point $y^*$. Theorem \ref{botest} implies, for any $j=1,\ldots, N$,
\begin{align}
\begin{split}
&C_s \cdot  N^{s/d}\geqslant \PP_s(A; N) \\
&\geqslant P_s(A; \omega_{N})=\sli_{x\in \omega_{N}} \frac{1}{|y^*-x|^s} \geqslant \frac{1}{|y^* - x_{j}|^s} = N^{s/d}\cdot (N^{1/d}\cdot|y^*-x_{j}|)^{-s};
\end{split}
\end{align}
therefore, $|y^*-x_j|\geqslant C_s^{-1/s} \cdot N^{-1/d}=:c_sN^{-1/d}$. 
\end{proof}

To prove Proposition \ref{babyhole} for the case $A=\S^d$ and $s\in [d-1, d)$ we set 
$$
U(y)=U_s(y):=\frac{1}{\H_d(\S^d)}\int_{\S^d} \frac{\textup{d}\H_d(x)}{|x-y|^s}.
$$
Then it is well known (see, e.g., \cite{Landkof1972}) that if $s\in (0,d)$ then $U(y)$ is constant of $\mathbb{S}^d$, and we denote this constant by $\gamma_{s,d}$ \footnote{$\gamma_{s,d}$ is the Wiener constant (maximal $s$-energy constant) on $\mathbb{S}^d$.}. 

We need the following lemma, which can be found in \cite{Kuijlaars2007}.
\begin{lemma}
For each $s\in [d-1, d)$ there exists a constant $C=C(s,d)$ such that for every $y$ with $|y|=1+N^{-1/d}$ we have 
\begin{equation}
U(y)\geqslant \gamma_{s,d}-CN^{-1+s/d}.
\end{equation}
Furthermore, if for a constant $c$ and an $N$-point configuration $\omega_N\subset \mathbb{S}^d$ we have $U(y)\leqslant c\cdot U^{\omega_N}(y)$, where
$$
U^{\omega_N}(y)=U^{\omega_N}_s(y) := \frac{1}N\sli_{x\in \omega_N} \frac{1}{|x-y|^s},
$$
then the same inequality holds for every $y\in \R^{d+1}$.
\end{lemma}
\begin{proof}[Proof of Proposition \ref{babyhole} for $A=\S^d$ and $s\in [d-1, d)$]
Fix an $N$-point configuration $\omega_N=\{x_1, \ldots, x_N\}$ and set $\gamma:=P_s(\S^d; \omega_N)$. For every $y\in \S^d$ we have
$$
U^{\omega_N}(y)\geqslant \frac{\gamma}N = \frac{\gamma}{\gamma_{s,d}\cdot N} \cdot U(y);
$$
thus, for every $y$ with $|y|=1+N^{-1/d}$ we have
$$
U^{\omega_N}(y)\geqslant \frac{\gamma}{\gamma_{s,d}\cdot N} \cdot (\gamma_{s,d}-CN^{-1+s/d}) = \frac{\gamma - C_1 \cdot \gamma \cdot N^{-1+s/d}}{N}.
$$
Notice that 
$$
\gamma=\inf_{y\in \S^d}\sli_{j=1}^N \frac{1}{|x_j-y|^s} \leqslant \frac{1}{\H_d(\mathbb{S}^d)}\sli_{j=1}^N \ili_{\S^d}\frac{\textup{d}\H_d(y)}{|x_j-y|^s} = \gamma_{s,d}\cdot N,
$$
which implies that for every $y$ with $|y|=1+N^{-1/d}$, we have
\begin{equation}\label{korol}
\sli_{j=1}^N\frac{1}{|x_j-y|^s} = NU^{\omega_N}(y)\geqslant \gamma - C_2 N^{s/d}.
\end{equation}
With $y^*$ as in the statement of Proposition \ref{babyhole}, set $y:=(1+N^{-1/d})\cdot y^*$. Then for every $j=1,\ldots, N$ we have
$|x_j-y|\geqslant |x_j-y^*|$. Therefore, for every $i=1,\ldots, N$, if follows from \eqref{korol} that
$$
\gamma-C_2 N^{s/d} - \frac{1}{|y-x_i|^s} \leqslant \sli_{j\not = i}\frac{1}{|y-x_j|^s} \leqslant \sli_{j\not=i} \frac{1}{|y^*-x_j|^s} = \gamma - \frac{1}{|y^*-x_i|^s}.
$$
We now use that $|x_i-y|\geqslant N^{-1/d}$ to get
$$
\frac{1}{|y^*-x_i|^s}\leqslant (C_2+1)N^{s/d},
$$
which completes the proof. 
\end{proof}
\section{Appendix: equivalent definition of best covering of the Euclidean space $\R^d$}\label{appendix}
Assume $\mathcal{B}\subset \R^d$ is a family of unit balls. The density of $\mathcal{B}$ is defined by
\begin{equation}\label{limdensity}
\Delta(\mathcal{B}):=\lim_{R\to \infty} \frac{\sum_{B\in \mathcal{B}} \H_d(B\cap [-R,R]^d)}{(2R)^d}
\end{equation}
whenever the limit exists. The optimal covering density for $\R^d$ is defined by 
$$
\Gamma_d:=\inf \Delta(\mathcal{B}),
$$
where the infimum is taken over all families $\mathcal{B}$ that cover $\R^d$.

It is known, see \cite[Chapter 2]{ConwSlBook} and \cite{Borodachov2016}, that $\Gamma_1$ is attained for balls centered on the lattice $2\mathbb{Z}$ and $\Gamma_{2}$ is attained for balls centered on the properly rescaled equi-triangular lattice. For higher dimensions no explicit results are known; however, if we minimize only over lattices, then it is known that for $d\leqslant 5$ an optimal lattice is the properly rescaled $A_d:=\{(x_1, \ldots, x_{d+1})\in \mathbb{Z}^{d+1}\colon x_1+\cdots+x_{d+1}=0\}$, which is a lattice in a $d$-dimensional hyperplane. 

We start by proving the following lemma.
\begin{lemma}\label{lemmatratata}
If $V_d=\H_d(\mathbb{B}^d)$, $\mathcal{B}$ covers $\R^d$ and the limit \eqref{limdensity} exists, then 
$$
\frac{\Delta(\mathcal{B})}{V_d}=\lim_{R\to \infty} \frac{\#\left\{B\in \mathcal{B}\colon \text{center of $B$ is in $[-R,R]^d$}\right\}}{(2R)^d}.
$$
Conversely, if the limit in the right-hand side exists, then $\Delta(\mathcal{B})$ exists as well and $\Delta(\mathcal{B})/V_d$ is equal to this limit.
\end{lemma}
\begin{proof}
Define $\mathcal{B}_R:=\left\{B\in \mathcal{B}\colon \mbox{center of $B$ is in $[-R,R]^d$}\right\}$. We estimate
\begin{equation}\label{tutuf}
\sum_{B\in \mathcal{B}} \H_d(B\cap [-R,R]^d)\geqslant \sum_{B\in \mathcal{B}_{R-2}} \H_d(B\cap [-R,R]^d) = V_d \cdot \#\mathcal{B}_{R-2}.
\end{equation}
On the other hand, if $B\cap [-R,R]^d \not=\emptyset$, then the center of $B$ is in $[-R-2, R+2]^d$. Therefore,
\begin{equation}\label{multik}
\sum_{B\in \mathcal{B}} \H_d(B\cap [-R,R]^d) \leqslant \sum_{B\in \mathcal{B}_{R+2}}\H_d(B\cap [-R,R]^d) \leqslant V_d \cdot \#\mathcal{B}_{R+2}.
\end{equation}
Estimates \eqref{tutuf} and \eqref{multik} obviously imply assertion of the lemma.
\end{proof}
We continue with more equivalent definitions of $\Gamma_d$. For a compact set $A\subset \R^d$ and a positive number $r$ put 
$$
N_A(r):=\min\Big\{N\in \mathbb{N}\colon \exists \omega_N=\{x_1, \ldots, x_N\}\subset A \; \mbox{such that} \; A\subset \cup_{j=1}^N B(x_j, r)\Big\}.
$$
A simple rescaling argument yields for every $R>0$
$$
N_{[-R, R]^d}(1) = N_{[0,1]}(1/2R).
$$
We show the following.
\begin{theorem}
For every $d \in \mathbb{N}$ we have
\begin{equation}\label{tyanitolkay}
\frac{\Gamma_d}{V_d}=\lim_{R\to \infty} \frac{N_{[-R, R]^d}(1)}{(2R)^d} = \lim_{r\to 0} r^d N_{[0,1]^d}(r) = \lim_{N\to \infty} N\cdot \rho_{[0,1]^d}(N)^d = \lim_{s\to \infty}(\sigma_{s,d})^{-d/s}.
\end{equation}
\end{theorem}
\begin{proof}
The existence of 
$$
\lim_{N\to \infty} N\cdot \rho_{[0,1]^d}(N)^d
$$
as well as the last equality follows from Theorem \ref{bestcovering}. The equalities
$$
\lim_{R\to \infty} \frac{N_{[-R, R]^d}(1)}{(2R)^d} = \lim_{r\to 0} r^d N_{[0,1]^d}(r) = \lim_{N\to \infty} N\cdot \rho_{[0,1]^d}(N)^d
$$
are straightforward and left to the reader. We derive the first equality in \eqref{tyanitolkay}. For a small $\ep>0$ take a set $\mathcal{B}$ such that 
$$
\frac{\Gamma_d}{V_d}\geqslant \lim_{R\to \infty}\frac{\#\mathcal{B}_R}{(2R)^d} - \ep
$$
and 
$$
\R^d = \bigcup_{B\in \mathcal{B}} B,
$$ 
where $\mathcal{B}_R$ is defined as in preceding proof. As in the proof of Lemma \ref{lemmatratata}, we have
$$
[-(R-2), R-2]^d\subset \bigcup_{B\in \mathcal{B}_R} B;
$$
therefore
$$
\frac{N_{[-(R-2), R-2]^d]}(1)}{(2(R-2))^d}\leqslant \frac{\#\mathcal{B}_R}{(2R)^d}\cdot \frac{(2R)^d}{(2(R-2))^d}.
$$
Consequently,
$$
\lim_{R\to \infty} \frac{N_{[-R, R]^d}(1)}{(2R)^d} \leqslant \frac{\Gamma_d}{V_d}+\ep.
$$
In view of the arbitrariness of $\ep$, we get 
\begin{equation}\label{krokodil}
\lim_{R\to \infty} \frac{N_{[-R, R]^d}(1)}{(2R)^d} \leqslant \frac{\Gamma_d}{V_d}.
\end{equation}

To prove the opposite inequality, we fix a large number $R_0$ and choose a configuration $\omega$ with $\#\omega = N_{[-R_0, R_0]^d}(1)$ and
$$
[-R_0,R_0]^d\subset \bigcup_{x\in \omega}B(x, 1).
$$
Define
$$
\mathcal{B}:=\{B(x,1)\colon  x\in((2R_0\mathbb{Z}^d)+\omega)\};
$$
then obviously
$$
\R^d=\bigcup_{B\in \mathcal{B}} B.
$$
Fix a number $R>R_0$ and choose an integer $n$ such that $(2n-1)R_0\leqslant R \leqslant (2n+1)R_0$. Then
$$
\#\mathcal{B}_{(2n-1)R_0}\leqslant \#\mathcal{B}_R\leqslant \#\mathcal{B}_{(2n+1)R_0}.
$$
Since
$$
\#\mathcal{B}_{(2n-1)R_0} = (2n-1)^d N_{[-R_0, R_0]^d}(1)
$$
and
$$
\#\mathcal{B}_{(2n+1)R_0} = (2n+1)^d N_{[-R_0, R_0]^d}(1),
$$
we get
$$
\left(\frac{2n-1}{2n+1}\right)^d \cdot \frac{N_{[-R_0, R_0]^d}(1)}{(2R_0)^d} \leqslant \frac{\#\mathcal{B}_R}{(2R)^d} \leqslant \left(\frac{2n+1}{2n-1}\right)^d \cdot \frac{N_{[-R_0, R_0]^d}(1)}{(2R_0)^d}.
$$
Therefore, 
$$
\lim_{R\to\infty} \frac{\#\mathcal{B}_R}{(2R)^d} = \frac{N_{[-R_0, R_0]^d}(1)}{(2R_0)^d},
$$
which implies, in view of Lemma \ref{lemmatratata}, that
$$
\frac{\Gamma_d}{V_d}\leqslant \frac{N_{[-R_0, R_0]^d}(1)}{(2R_0)^d}.
$$
From of the arbitrariness of $R_0$ and the estimate \eqref{krokodil}, the lemma follows. 
\end{proof}
\bibliography{ref}
\bibliographystyle{plain}

\end{document}